\documentclass[11pt]{article}
\usepackage[utf8]{inputenc}
\usepackage{graphicx, amsfonts, amsmath, mathrsfs, amsthm, dsfont, mathtools, amssymb, tikz, tikz-cd, geometry, comment, titlesec, xcolor, pdfpages, adjustbox, titling}

\usetikzlibrary{calc, decorations.pathmorphing, decorations.markings}

\definecolor{red}{rgb}{1,0,0}
    
\newenvironment{myequation}
    {
        \begin{equation*}
        \begin{gathered}
    }
    { 
        \end{gathered} 
        \end{equation*}
    }

\newtheorem{theorem}{Theorem}[section]
\newtheorem{claim}[theorem]{Claim}
\newtheorem{lemma}[theorem]{Lemma}
\newtheorem{definition}[theorem]{Definition}
\newtheorem{prop}[theorem]{Proposition}
\newtheorem*{remark}{Remark}
\newtheorem*{theorem*}{Theorem}
\newtheorem*{corollary}{Corollary}
\newcommand{\R}{\mathbb{R}}
\newcommand{\N}{\mathbb{N}}
\newcommand{\Q}{\mathbb{Q}}
\newcommand{\Z}{\mathbb{Z}}

\DeclareMathOperator{\Ima}{Im}

\pretitle{\begin{center}\Huge\bfseries}
\posttitle{\par\end{center}\vskip 0.5em}
\preauthor{\begin{center}\Large}
\postauthor{\end{center}}
\predate{\par\large\centering}
\postdate{\par}

\title{Eggbeater dynamics on symplectic surfaces of genus 2 and 3}
\date{\today}
\author{Arnon Chor}

\begin{document}

\maketitle
\thispagestyle{empty}

\begin{abstract}
    The group $Ham(M,\omega)$ of all Hamiltonian diffeomorphisms of a symplectic manifold $(M,\omega)$ plays a central role in symplectic geometry. This group is endowed with the Hofer metric. In this paper we study two aspects of the geometry of $Ham(M,\omega)$, in the case where $M$ is a closed surface of genus 2 or 3. First, we prove that there exist diffeomorphisms in $Ham(M,\omega)$ arbitrarily far from being a $k$-th power, with respect to the metric, for any $k \geq 2$. This part generalizes previous work by Polterovich and Shelukhin. Second, we show that the free group on two generators embeds into the asymptotic cone of $Ham(M,\omega)$. This part extends previous work by Alvarez-Gavela et al. Both extensions are based on two results from geometric group theory regarding incompressibility of surface embeddings.
\end{abstract}

\section{Introduction and main results}
    \subsection{Introduction}
        Recall that a symplectic manifold is a smooth even-dimensional manifold $M$, equipped with a closed non-degenerate differential 2-form $\omega$. Non-degeneracy of $\omega$ means that its top wedge power, $\omega^n$, is a volume form. Symplectic manifolds serve as models of phase spaces of classical mechanics. The group of all automorphisms of a symplectic manifold, i.e. maps $M \to M$ which preserve $\omega$, contains a subgroup of all physically possible mechanical motions. This is the group of Hamiltonian diffeomorphisms, denoted $Ham(M,\omega)$, which plays a central role in symplectic topology and Hamiltonian dynamics.
        
        In 1990 H. Hofer introduced a remarkable bi-invariant Finsler metric on $Ham$ (see~\cite{H}). Studying coarse geometry of the metric space $Ham$ is an important problem of modern symplectic topology, and is still far from well-understood. This paper takes a step in this direction.
        
        Let us briefly outline our main results. Denote by $Powers_k \subset Ham$ the set of elements admitting a root of degree $k$. It has been shown in~\cite{PS14} that for surfaces of genus $\geq 4$ the complement $Ham \setminus Powers_k$ contains an arbitrarily large ball, with respect to the metric introduced by Hofer. We extend this result to surfaces of genus 2 and 3.
        
        Further, according to~\cite{CZ}, any asymptotic cone (in the sense of Gromov, see~\cite{G}) of a group with bi-invariant metric has a natural group structure. It has been shown in~\cite{10} that for a surface $M$ of genus $\geq 4$ any such cone of $Ham(M,\omega)$ contains a free group with 2 generators. We extend this result to surfaces of genus 2 and 3.
        
        The proofs of the above-mentioned results on surfaces of genus $\geq 4$ involve Floer homology of non-contractible closed orbits on the surface, and are proved by considering special elements of $Ham$, the so-called eggbeater maps (see~\cite{FO} and~\cite{PS14}), which originate in chaotic dynamics. An algebraic analysis of non-contractible closed orbits of these maps plays an important role in the proof. This is exactly the place where we had to modify the original arguments, in order to extend the results to the cases of genera 2 and 3. Our main innovations are two algebraic results about homomorphisms from a free group into a surface group (see Lemma~\ref{lemma:2} and Claim~\ref{cl:incompressible} below).
    
    \subsection{Preliminaries}
        Recall that Hofer's metric on $Ham(M,\omega)$ is defined as
        \begin{myequation}
            d_H(f,g) = \inf_H \int_0^1 \left( \max_M H_t - \min_M H_t \right) dt ,
        \end{myequation}
        for $f,g \in Ham(M,\omega)$, where the infimum is taken over all smooth $H: S^1 \times M \to \R$ that generate $f^{-1}g$ and where $H_t = H(t,\cdot)$. This is a bi-invariant metric on $Ham(M)$, and the fact that it is a genuine metric, as opposed to a pseudo-metric, is non-trivial (see~\cite{H},\cite{LM}). Hofer's norm of a diffeomorphism is its Hofer distance to the identity, and is denoted
        \[
            ||\cdot||_H = d_H(id,\cdot) .
        \]\\
        
        The focus of this paper is the metric space $(Ham(M), d_H)$. The question whether for every symplectic manifold $Ham(M)$ has infinite diameter with respect to Hofer's metric is an important open problem in this field; it is conjectured (see discussion in 14.2 of~\cite{MS}) that the answer to this question is positive for every closed symplectic manifold $M$. The conjecture has been partially confirmed (see~\cite{Sc},\cite{P},\cite{MS}). \\
        
        Rather than asking about the diameter of the group $Ham(M)$, one can ask about the supremum of distances $d_H(f,X)$ for some subset $X \subset Ham(M)$. A natural set $X$ to ask this question for is $Aut(M,\omega)$, the set of autonomous (i.e. "time-independent") Hamiltonian diffeomorphisms, and another interesting family of sets are the sets $Powers_k = \{\psi^k | \psi \in Ham(M,\omega)\}$ of Hamiltonian diffeomorphisms which admit $k^{th}$ roots, for $k \geq 2$ an integer. The quantities queried are the following:
        \begin{myequation}
            aut(M,\omega) = \sup_{\phi \in Ham(M, \omega)} d_H(\phi, Aut(M, \omega)) ,\\
            powers_k(M,\omega) = \sup_{\phi \in Ham(M, \omega)} d_H(\phi, Powers_k(M,\omega)) .
        \end{myequation}
        Note that for a symplectic manifold $M$, showing that $aut(M) = \infty$ or that for any $k \geq 2$, $powers_k(M) = \infty$ would answer the Hamiltonian diameter question for $M$.
        
        L. Polterovich and E. Shelukhin conjectured in~\cite{PS14} that $aut(M) = \infty$ for all closed symplectic manifolds, and made a first step in that direction: they show that symplectic surfaces $M$ of genus $\geq 4$ have  $powers_k(M) = \infty$ for all $k \geq 2$. One of the results in this paper states this is also true for symplectic surfaces of genera 2,3 (see Theorem~\ref{thm:2} below). \\
        
        Our second result concerns the coarse structure of the metric space $(Ham(M), d_H)$. To state it we need the notions of the asymptotic cone of a metric space, which is an important notion in coarse geometry (see~\cite{G}), and of ultrafilters and ultralimits. A \textit{filter} on a partially ordered set $(P,\leq)$ is a non-empty proper subset $F \subset P$ that is upward closed and downward directed; i.e. if $x \in F, y \in P, x \leq y$ then $y \in F$, and also $\forall x,y \in F \ \exists z \in F$ such that $z \leq x,y$. A \textit{non-principal ultrafilter} on $(P,\leq)$ is a filter $F$ on $(P,\leq)$ such that there is no filter $F^\prime$ on $P$ with $F \subset F^\prime \subset P$, and such that $F$ is not of the form $\{x \in P | y \leq x\}$ for any $y \in P$. Given a metric space $(X,d)$, an ultrafilter $\mathcal{U}$ on the power set of the natural numbers $2^\N$ (equipped with the inclusion order) and a sequence of points $(x_n)$ in $X$, a point $x \in X$ is the $\mathcal{U}$-\textit{ultralimit} of $(x_n)$, denoted $\lim_\mathcal{U} x_n$, if for any $\epsilon > 0$, $\{n | d(x_n,x) \leq \epsilon\} \in \mathcal{U}$. The ultralimit does not necessarily exist, but can be shown to exist if $(x_n)$ is bounded.
        
        Let $(X,d)$ be a metric space, fix $\mathcal{U}$ a non-principal ultrafilter on $2^\N$, and fix some basepoint $x_0 \in X$. The \textit{asymptotic cone} of $(X,d)$ is a metric space $Cone_\mathcal{U}(X,d)$ whose underlying set is 
        \begin{myequation}
            \left\{ \left( x_k \right)_{k \in \N} \in X^\N \middle| \exists C > 0 \ s.t.\  \forall k: \frac{d(x_k,x_0)}{k} < C \right\} \bigg/ \thicksim ,
        \end{myequation}
        where $(x_k) \sim (y_k)$ if $\lim_\mathcal{U} \frac{d(x_k,y_k)}{k} = 0$, and whose metric is 
        \begin{myequation}
            d_\mathcal{U} ([(x_k)], [(y_k)]) = \lim_\mathcal{U} \frac{d(x_k,y_k)}{k} .
        \end{myequation}
        
        Assume additionally that $X$ is a group, and that $d$ is a bi-invariant metric. Then $Cone_\mathcal{U}(X,d)$ is also a group, with multiplication
        \begin{myequation}
            [(x_k)] \cdot [(y_k)] = [(x_k \cdot y_k)] .
        \end{myequation}
        Since $d$ is bi-invariant, this multiplication is well defined, and $d_\mathcal{U}$ is also bi-invariant.
        
        Elements of the asymptotic cone represent directions (or rather, velocities) in which one can go to infinity in the base space $X$. For example, bounded metric spaces all have asymptotic cones which are single points, and on the other hand, the asymptotic cone of the hyperbolic plane is a tree with uncountably many branches at each point. The asymptotic cone is an invariant of the coarse structure, or the large-scale properties of a metric space, in the sense that quasi-isometric spaces have the same asymptotic cones (see~\cite{R} for more on coarse structure, asymptotic cones, and quasi-isometry). \\
        
        The focus of this paper is the geometry of $Ham(M)$ for $(M,\omega)$ a symplectic manifold, therefore we consider $Cone_\mathcal{U}(Ham(M), d_H)$ for a non-principal ultrafilter $\mathcal{U}$ on $2^\N$. In~\cite{10}, D. Alvarez-Gavela et al. show that given a symplectic surface $M$ of genus $\geq 4$, there exists a monomorphism $F_2 \to Cone_\mathcal{U}(Ham(M))$, where $F_2$ is the free group on two generators, and therefore $Cone_\mathcal{U}(Ham(M))$ has a subgroup isomorphic to $F_2$. The second result in this paper states this is also true for symplectic surfaces of genera 2,3 (see Theorem~\ref{thm:3} below).
        
        We turn now to our main results.
    
    \subsection{Results}
    \label{sec:results}
        The following results are generalizations of previous theorems appearing in~\cite{PS14},~\cite{10}: specificaly, Theorem~\ref{thm:2} is a generalization of Theorem 1.3 in~\cite{PS14}, and Theorem~\ref{thm:3} is a generalization of Theorem 1.1 in~\cite{10}. The original theorems are the same as stated here, except for their assumptions on the surface $\Sigma$: while the original theorems hold for all closed symplectic surfaces $\Sigma$ of genus $\geq 4$, the theorems presented here hold for closed symplectic surfaces of genera 2 and 3.
        
        \begin{theorem}
        \label{thm:2}
            Let $\Sigma$ be a closed oriented surface of genus 2 or 3, equipped with an area form $\sigma$, and $k \geq 2$ an integer. Then $powers_k(\Sigma, \sigma) = \infty$.
        \end{theorem}
        
        \begin{theorem}
        \label{thm:3}
            Let $\Sigma$ be a closed oriented surface of genus 2 or 3, equipped with an area form $\sigma$. Then for any non-principal ultrafilter $\mathcal{U}$ on $2^\N$, there exists a monomorphism $F_2 \hookrightarrow Cone_\mathcal{U}(Ham(\Sigma), d_H)$.
        \end{theorem}
        
        We remark that since $aut(M) \geq powers_k(M)$, Theorem~\ref{thm:2} implies that $aut(\Sigma, \sigma) = \infty$ for any closed oriented surface $\Sigma$ of genus 2 or 3 with an area form $\sigma$. We remark further that the above results survive stabilization by a closed aspherical symplectic manifold. That is, if $(M,\omega)$ is a symplectic manifold with $\pi_2(M) = 0$ and $\Sigma$ is as above, then the results also hold for the symplectic manifold $(\Sigma \times M, \sigma \oplus \omega)$. This is shown in the same way as in~\cite{PS14},\cite{10}.
        
        In the proofs of our results we closely follow~\cite{PS14} and~\cite{10}. Specifically, we use the same construction called eggbeater maps (see~\cite{PS14} and~\cite{FO}). We will outline the proofs of the original theorems, and give an in-depth explanation of the changed parts in Section~\ref{sec:3}. An alternative proof of the theorems in genus 3 is presented in Section~\ref{sec:incompressibility}.

    \subsection{Outline of the proofs for genera 2,3}
    \label{sec:outline}
        In this subsection we present a short outline of the proofs of the theorems. The full details of the proofs are given in Section~\ref{sec:3}. \\
        
        \begin{figure}[!ht]
            \centering
            \begin{minipage}{.5\textwidth}
                \begin{minipage}{\textwidth}
                    \scalebox{0.8}{
                    \begin{tikzpicture}
                        \def\a{70};
                        \def\b{55};
                        \def\c{83};
                        \def\d{66.5};
                        
                        \draw (-1,0) ++(\a:3) arc (\a:360-\a:3);
                        \draw (-1,0) ++(\b:3) arc (\b:-\b:3);
                        \draw (-1,0) ++(\c:2.5) arc (\c:360-\c:2.5);
                        \draw (-1,0) ++(\d:2.5) arc (\d:-\d:2.5);
                        
                        \draw (1,0) ++(180-\a:3) arc (540-\a:180+\a:3);
                        \draw (1,0) ++(180-\b:3) arc (180-\b:180+\b:3);
                        \draw (1,0) ++(180-\c:2.5) arc (540-\c:180+\c:2.5);
                        \draw (1,0) ++(180-\d:2.5) arc (180-\d:180+\d:2.5);
                    
                        \draw (-1.5,3) node[anchor=south] {\Large $C_V$};
                        \draw (1.5,3) node[anchor=south] {\Large $C_H$};
                    \end{tikzpicture}}
                    \caption{The manifold $C$ used in the proof.}
                    \label{fig:C}
                \end{minipage} \\[1.5\baselineskip]
                \begin{minipage}{\textwidth}
                    \scalebox{0.9}{
                    \begin{tikzpicture}
                        \usetikzlibrary{decorations.markings}
                        \begin{scope}
                        [decoration={markings,mark=at position 0.55 with {\arrow[scale=2]{>}}}]
                            \draw
                                (-3,1) -- (3,1) -- (3,-1)
                                decorate {(3,1) -- (3,-1)};
                            \draw
                                (-3,1) -- (-3,-1) -- (3,-1)
                                decorate {(-3,1) -- (-3,-1)};
                        \end{scope}
                        
                        \draw (-2,1) -- (2,0) -- (-2,-1);
                        \draw[dashed] (-2,1) -- (-2,-1);
                        
                        \draw[->] (-2,0.6) -- (-0.9,0.6);
                        \draw[->] (-2,-0.6) -- (-0.9,-0.6);
                        \draw[->] (-2,0.2) -- (0.8,0.2);
                        \draw[->] (-2,-0.2) -- (0.8,-0.2);
                    \end{tikzpicture}}
                    \caption{The profile of the map $f: C_* \to C_*$.}
                    \label{fig:profile}
                \end{minipage}
            \end{minipage}%
            \hspace{0.1cm}
            \begin{minipage}{.48\textwidth}
                \centering
                \scalebox{0.9}{
                \begin{tikzpicture}
                    \draw (-2.5,-1) -- (-1,-1) -- (-1,-4.5);
                    \draw (2.5,-1) -- (1,-1) -- (1,-4.5);
                    \draw (2.5,1) -- (1,1) -- (1,2.5);
                    \draw (-2.5,1) -- (-1,1) -- (-1,2.5);
                    \draw[dotted] (-3,1) -- (-2.5,1);
                    \draw[dotted] (1,-4.5) -- (1,-4.8);
                    \draw[dotted] (1,2.8) -- (1,2.5);
                    \draw[dotted] (3,1) -- (2.5,1);
                    \draw[dotted] (3,-1) -- (2.5,-1);
                    \draw[dotted] (-1,2.8) -- (-1,2.5);
                    \draw[dotted] (-1,-4.5) -- (-1,-4.8);
                    \draw[dotted] (-3,-1) -- (-2.5,-1);
                    \draw (3,0) node[anchor=south] {$C_H$};
                    \draw (0,-4.8) node[anchor=west] {$C_V$};
                    
                    \draw[dashed] (-2,1) -- (-2,-1);
                    
                    \draw (-2,1) -- (-1,0.75) -- (0,-3.5) -- (1,0.25) -- (2,0) -- (1,-0.25) -- (0,-4.5) -- (-1,-0.75) -- (-2,-1);
                    
                    \draw[->] (-2,0.6) -- (-1.6,0.6);
                    \draw[->] (-2,-0.6) -- (-1.6,-0.6);
                    \draw[->] (-2,0.2) -- (-1.1,0.2);
                    \draw[->] (-2,-0.2) -- (-1.1,-0.2);
                    \draw (-2,0) node[shape=circle,draw,outer sep=2pt,anchor=east] {1};
                    
                    \draw[->] (-0.6,2) -- (-0.6,1.6);
                    \draw[->] (0.6,2) -- (0.6,1.6);
                    \draw[->] (-0.2,2) -- (-0.2,1.1);
                    \draw[->] (0.2,2) -- (0.2,1.1);
                    \draw (0,2) node[shape=circle,draw,outer sep=2pt,anchor=south] {2};
                \end{tikzpicture}}
                \caption{The neighborhood of an intersection of $C_V$ and $C_H$, with the profile of $f_V \circ f_H$ pictured. The map $f_V \circ f_H$ is a shear along $C_H$ (1) and then a shear along $C_V$ (2).}
                \label{fig:profile2}
            \end{minipage}
        \end{figure}
        
        Denote by $C = C_V \bigcup C_H$ the union of two identical annuli (see Figure~\ref{fig:C}), and by $C_*$ a single such annulus. Let $f: C_* \to C_*$ be a piecewise-linear shear map along the axis of $C_*$, whose profile consists of two straight lines (see Figure~\ref{fig:profile}). Applying the map $f$ on $C_V,C_H$, one gets two shear maps $f_V,f_H: C \to C$ with support in $C_V,C_H$ respectively. Composing the maps $f_V,f_H$ and their inverses several times in some specific order yields a map $C \to C$ called an egg-beater map. Figure~\ref{fig:profile2} depicts the profile of such an egg-beater map $f_V \circ f_H$ in the neighborhood of one of the intersections of $C_V$ and $C_H$.
        
        The proofs of both results, Theorem~\ref{thm:2} and Theorem~\ref{thm:3}, are based on counting the fixed points of the egg-beater map in $C$ whose orbits belong to some suitably selected free homotopy classes $\alpha_k \in \pi_0(\mathscr{L}C)$ (where $\mathscr{L}X$ is the free loop space of a space $X$), and studying the Floer homology of these orbits. Using this tool, it can be shown that if there are not too many such fixed points and some condition on the actions of their orbits holds, then the theorems themselves hold. This is all done in~\cite{PS14} and~\cite{10}.
        
        However, in order to get results on a closed oriented surface $\Sigma_g$ of genus $g$, this construction must be embedded in $\Sigma_g$, using an embedding denoted $i: C \hookrightarrow \Sigma_g$, and a similar analysis done there. In this step, one might encounter a new problem: consider $\tau_i: \pi_0(\mathscr{L}C) \to \pi_0(\mathscr{L}\Sigma_g)$, the mapping induced by $i$. The mapping $\tau_i$ might not be injective. In this case, the previous analysis done on $C$ will no longer hold, and one must then re-count fixed points in $\Sigma_g$, since the number of fixed points of the egg-beater map on $\Sigma_g$ in class $\alpha \in \pi_0(\mathscr{L}\Sigma_g)$ is the number of fixed points of the egg-beater map on $C$ in all the classes $\tau_i^{-1}(\alpha)$, which might be too much for these methods to work.
        
        Originally, in genus $g \geq 4$, this problem did not arise, since if $g \geq 4$ one can find an embedding $i$ such that $\tau_i$ is injective. In fact, one can also find an embedding $i_3$ into $\Sigma_3$ with $\tau_{i_3}$ injective; this approach is shown in Section~\ref{sec:incompressibility}. It is likely that there also exists an embedding $i_2$ into $\Sigma_2$ with an injective $\tau_{i_2}$. However, this isn't proven in this paper. \\
        
        The results for genera 2 and 3 are proven in Section~\ref{sec:3} by placing bounds on the non-injectivity of $\tau_i$, using a wisely chosen $i$ and an algebraic lemma. More precisely, note that injectivity of $\tau_i$ can be translated to a claim on images of conjugacy classes of $\pi_1(C)$ under the homomorphism $i_*: \pi_1(C) \to \pi_1(\Sigma_2)$ induced by $i$. One may choose $i$ such that the induced homomorphism $i_*: F_3 = \langle a,b,c \rangle \to \langle g_1,g_2,g_3,g_4 | [g_1,g_2][g_3,g_4] \rangle$ is the following:
        \begin{myequation}
            a \mapsto g_1 g_3 ,\\
            b \mapsto g_2 g_1^{-1} g_2^{-1} g_3 ,\\
            c \mapsto g_3 .
        \end{myequation}
        
        The main novel ingredient of the results of this paper is the following lemma:
        
        \begin{lemma}
        \label{lemma:2}
            Let $p \in \Z_{\geq 1}$. For all $1 \leq j \leq p$, let $k_j,l_j \in \Z, u_j,v_j \in \{0,1\}, 0 \neq m_j,n_j \in \Z$. Consider the homomorphism $i_*: F_3 \to \pi_1(\Sigma_2)$ given above. Let $\delta = \Pi_j a^{k_j}c^{-u_j}b^{l_j}c^{v_j} \in F_3$ and $\beta = \Pi_j a^{m_j}b^{n_j} \in F_3$.
            
            If $i_* \delta, i_* \beta$ are conjugates (in $\pi_1(\Sigma_2)$), then so are $\delta, \beta$ (in $F_3$).
        \end{lemma}
        
        This lemma gives bounds on the non-injectivity of $\tau_i$: different conjugacy classes having a specific form cannot have the same image under $\tau_i$. One may show that all orbits of the egg-beater map have free homotopy classes of this specific form (see Claim~\ref{claim:3}). This allows one to count fixed points of the egg-beater map in $\Sigma_g$ whose orbits are in specific free homotopy classes $\alpha_k$ by counting the corresponding fixed points in $C$ whose orbits are in $\tau_i^{-1}(\alpha_k)$. This calculation was carried out in~\cite{PS14} and~\cite{10}, and shows that there are not too many such fixed points, and that the condition on their actions mentioned above holds. Therefore, using the Floer homology tool, one can prove Theorem~\ref{thm:2} and Theorem~\ref{thm:3}. \\
        
        A different approach is shown in Section~\ref{sec:incompressibility}: instead of bounding the non-injectivity of $\tau_i$, a different embedding $i_3$ into a surface of genus 3 is defined. The embedding $i_3$ has a restriction $i_3 \restriction_C$ that can be shown to induce an injective $\tau_{i_3 \restriction_C}$ with an argument using the intersection number of free homotopy classes (see Claim~\ref{cl:incompressible}). With this injectivity in hand, the above discussion yields the desired result on the number of fixed points of the egg-beater map. \\
    
        Section~\ref{sec:2} contains the proof of Lemma~\ref{lemma:2} and a related result. The proofs of the theorems for genera 2 and 3, using the lemma, can be found in Section~\ref{sec:3}. Proofs of the theorems for genus 3, using injectivity of the induced map $\tau_{i_3}$ of the suitably-defined embedding $i_3$, are found in Section~\ref{sec:incompressibility}.

    \subsection{Acknowledgements}
        This paper is a part of the my M.Sc thesis, carried out under the supervision of Prof. Leonid Polterovich and Prof. Yaron Ostrover, whom I would like to sincerely thank for their contributions and knowledge. I would like to thank Matthias Meiwes and Leonid Potyagailo, for many fruitful discussions and ideas. I would also like to thank Ofir Karin, Leonid Vishnevsky and Asaf Cohen, for helpful comments and remarks. This work was partially supported by ISF grant numbers 1274/14 and 667/18.

\section{Proof of the lemma}\label{sec:2}
    Recall the notation in chapter IV of~\cite{LS}. Let $G,H$ be finitely presentable groups, $A < G, B < H$ be isomorphic subgroups, and $\psi: A \xrightarrow{\sim} B$ an isomorphism. The \textit{free product of $G,H$ with respect to $\psi$} (or free product of $G,H$ with amalgamation), denoted $\langle G * H, A = B, \psi \rangle$, is defined as follows. If $G = \langle S_1 | R_1 \rangle, H = \langle S_2 | R_2 \rangle$ are finite presentations with $S_1 \cap S_2 = \emptyset$, then the free product with amalgamation is defined to be
    \[
        \langle G * H, A = B, \psi \rangle = \langle S_1 \cup S_2 | R_1, R_2, \{a \psi(a)^{-1} | a \in A\} \rangle .
    \]
    The groups $G$ and $H$ are called the \textit{factors} of $\langle G * H, A = B, \psi \rangle$.
    
    Free products with amalgamation occur naturally in topology: let $X$ is a topological space with open cover $\{Y,Z\}$ such that $Y \cap Z$ is connected, and denote $G = \pi_1(Y), H = \pi_1(Z)$; $A = \pi_1(Y \cap Z)$, considered as a subgroup of $G$; and $B = \pi_1(Y \cap Z)$, considered as a subgroup of $H$. Denote also by $\psi: A \xrightarrow{\sim} B$ the natural isomorphism. Then by the van Kampen theorem:
    \[
        \pi_1(X) = \langle G * H, A = B, \psi \rangle .
    \]
    
    Free products with amalgamation have a certain uniqueness property of conjugacy classes, which will be stated soon. First, we must recall the definition of cyclically reduced elements.
    
    \begin{definition}
        A sequence $c_1,...,c_n$ (with $n \geq 0$) of elements of $\langle G*H, A=B, \psi \rangle$ is called reduced if:
        \begin{enumerate}
            \item Each $c_i$ is in one of the factors $G$ or $H$.
            \item Successive $c_i,c_{i+1}$ come from different factors.
            \item If $n > 1$, no $c_i$ is in $A$ or $B$.
            \item If $n = 1$, $c_1 \neq 1$.
        \end{enumerate}
        
        A sequence $c_1,...,c_n$ of elements of $\langle G * H, A = B, \phi \rangle$ is called cyclically reduced if all its cyclic permutations (i.e. $c_2,...,c_n,c_1$, etc.) are reduced.
        
        An element $u \in \langle G * H, A = B, \phi \rangle$ is called cyclically reduced if there exists a cyclically reduced sequence $c_1,...,c_n$ such that $u = c_1 \cdot ... \cdot c_n$ (in this case the sequence $(c_i)_{i=1}^n$ is said to represent $u$).
    \end{definition}
    
    We remark that every element of $\langle G * H, A = B, \phi \rangle$ is conjugate to a (not necessarily unique) cyclically reduced element. Recall the following theorem (Theorem 2.8 in chapter IV of~\cite{LS}):
    
    \begin{theorem*}[Conjugacy Theorem for Free Products with Amalgamation]
        Let $P = \langle G * H, A = B, \phi \rangle$ be a free product with amalgamation. Let $u \in P$ be a cyclically reduced element, and let $c_1, ..., c_n$ be any cyclically reduced sequence with $u = c_1 \cdot ... \cdot c_n$ where $n \geq 2$. Then every cyclically reduced conjugate of $u$ can be obtained by cyclically permuting $c_1 \cdot \cdot \cdot c_n$ and then conjugating by an element of the amalgamated part~$A$: if $v \in P$ is cyclically reduced and conjugate to $u$, then $\exists 0 \leq k < n$ and $\exists a \in A$ such that $v = a \cdot c_k \cdot ... \cdot c_n \cdot c_1 \cdot ... \cdot c_{k-1} a^{-1}$.
    \end{theorem*}
    
    Denote the conjugacy relation in a group by $\sim$. For any group $G$, denote the conjugacy class of an element $x \in G$ by $[x]_G$.

    The Conjugacy Theorem implies one can define a length on elements of $P = \langle G*H , A=B, \phi \rangle$ by
    \begin{myequation}
        len: P \to \Z_{\geq 0} ,\\
        u \mapsto n ,
    \end{myequation}
    where $n$ is the length of a cyclically reduced sequence $c_1,...,c_n$ such that $u \sim \Pi_j c_j$. This is well defined by the Conjugacy Theorem, and is obviously conjugation-invariant. Note that $len(u) = 0 \iff u = 1$ and $len(u) = 1$ if and only if $u$ is conjugate to an element of $G \cup H \subset P$. \\
    
    Let us denote the following groups:
    \begin{myequation}
        H_1 = \langle g_1, g_2 \rangle \simeq F_2 , \ H_2 = \langle g_3, g_4 \rangle \simeq F_2 ,\\
        A = \langle [g_1,g_2] \rangle < H_1 , \ B = \langle [g_3,g_4] \rangle < H_2 ,
    \end{myequation}
    and denote by $\phi: A \to B$ the isomorphism $[g_1,g_2] \mapsto [g_3,g_4]^{-1}$. Let $F_3 = \langle a,b,c \rangle$ be the free group with three generators, and let $\pi_1(\Sigma_2) = \langle g_1,g_2,g_3,g_4 | [g_1,g_2][g_3,g_4 \rangle = \langle H_1 * H_2, A = B, \phi \rangle$ be the first homotopy group of a closed oriented surface of genus 2.
    
    Consider the homomorphism $\varphi: F_3 \to \pi_1(\Sigma_2)$ defined by
    \begin{myequation}
        a \mapsto g_1 ,\\
        b \mapsto g_2 g_1 g_2^{-1} ,\\
        c \mapsto g_3 .
    \end{myequation}
    
    The next claim gives a restriction on non-injectivity of conjugacy classes by $\varphi$.
    
    \begin{claim}\label{cl:1}
    Let $r,s \in F_3$, and assume that $\varphi(r) \sim \varphi(s)$ in $\pi_1(\Sigma_2)$ and $r \not\sim s$ in $F_3$. Then exactly one of the following holds:
    \begin{itemize}
        \item $\exists 0 \neq j \in \Z$ such that $\varphi(r),\varphi(s)$ are conjugate to $g_1^j$ (in $\pi_1(\Sigma_2)$).
        \item $\exists 0 \neq j \in \Z$ such that $\varphi(r),\varphi(s)$ are conjugate to $g_3^j$ (in $\pi_1(\Sigma_2)$).
    \end{itemize}
    
    In other words, the only conjugacy classes in $F_3$ merged by the homomorphism $\varphi$ are the classes $[c^j]_{F_3}$ which merge with $[(ab^{-1}c)^j]_{F_3}$, and $[a^j]_{F_3}$ which merge with $[b^j]_{F_3}$ ($0 \neq j \in \Z$). These correspond to the following conjugacy classes in $\pi_1(\Sigma_2)$: $[g_3^j]_{\pi_1(\Sigma_2)} = [g_4 g_3^j g_4^{-1}]_{\pi_1(\Sigma_2)}$ and $[g_1^j]_{\pi_1(\Sigma_2)} = [g_2 g_1^j g_2^{-1}]_{\pi_1(\Sigma_2)}$.
    \end{claim}
    
    \begin{proof}
        Note that $\varphi(a)$ and $\varphi(b)$ are both elements in the same factor $H_1$, and $\varphi(c) \in H_2$. Partition $r$ and $s$ into sequences according to the partition $\{a,b,c\} = \{a,b\} \cup \{c\}$, i.e. concatenate consecutive symbols from $\langle a,b \rangle$, then perform the following steps, until no steps can be performed:
        
        \begin{itemize}
        \item 
            Concatenate elements of the sequence of the form $(ab^{-1})^j$ (for some $j \in \Z$) to the previous and next elements in the sequence; do this for all occurrences of $(ab^{-1})^j$. I.e, if the sequence is $g_1, h_1, ab^{-1}, h_2, (ab^{-1})^{-2}, h_3$ (with all $g_i \in \langle a,b \rangle, h_i \in \langle c \rangle$), the resulting sequence after this step will be $g_1, h_1 ab^{-1} h_2 (ab^{-1})^{-2} h_3$.
        \item
            If the first and last elements of the sequence are from the same factor (in $\{a,b\},\{c\}$), concatenate them cyclically: i.e, if the sequence is $ab, c, b$, the resulting sequence after this step will be $bab, c$.
        \end{itemize}
        
        Doing this results in sequences $(r_i)_{i=1}^n,(s_i)_{i=1}^m$ such that $r \sim \Pi_i r_i, s \sim \Pi_i s_i$ (in $F_3$) and the sequences $(\varphi(r_i)), (\varphi(s_i))$ are cyclically reduced in $\pi_1(\Sigma_2)$ (since the only way to generate an element of $A$ or $B$ from images of $a,b,c$ is $\varphi(ab^{-1})^j = [g_1,g_2]^j \in A$; this can be seen from the definition of $\varphi$). For example, if $r$ were the element $abcba^{-1}cac^{-1}b$, we first partition according to $\{a,b\},\{c\}$ to get $ab, c, ba^{-1}, c, a, c^{-1}, b$, then concatenate powers of $ab^{-1}$ to get $ab, cba^{-1}c, a, c^{-1}, b$, and finally concatenate the first and last elements to get the sequence $(r_i) = bab, cba^{-1}c, a, c^{-1}$. Note that the resulting sequence is not uniquely defined, but any sequence which is the result of these steps will do for our purposes. \\
        
        Now $(\varphi(r_i)), (\varphi(s_i))$ are cyclically reduced sequences of $\pi_1(\Sigma_2)$, and $\Pi_{i=1}^m \varphi(s_i) \sim \varphi(s) \sim \varphi(r) \sim \Pi_{i=1}^n \varphi(r_i)$. Assume $n \geq 2$, we will reach a contradiction. By the Conjugacy Theorem for Free Products with Amalgamation, $\exists \alpha \in A, 0 \leq k < n$ such that
        \begin{equation}
        \label{eq:1}
            \Pi_i \varphi(s_i) = \alpha \cdot \varphi(r_k) \varphi(r_{k+1}) ... \varphi(r_n) \varphi(r_1) ... \varphi(r_{k-2}) \varphi(r_{k-1}) \cdot \alpha^{-1} .
        \end{equation}
        
        Since $A \subset \Ima \varphi$, $\alpha = \varphi(\sigma)$ for some $\sigma \in F_3$. Therefore, pulling back Equation~\ref{eq:1} through $\varphi$, we get the following equation in $F_3$:
        \begin{myequation}
            \Pi_i s_i = \sigma \cdot r_k r_{k+1} ... r_{k-2} r_{k-1} \sigma^{-1} .
        \end{myequation}
        Note that this can be done since $\varphi$ is a monomorphism.
        
        Now, it can be seen that $s \sim \Pi_i s_i = \sigma r_k r_{k+1} ... r_{k-2} r_{k-1} \sigma^{-1} \sim r$ in $F_3$, contradicting our assumption. Therefore, $n \leq 1$. By symmetry of the above argument with respect to $r$ and $s$, $m \leq 1$ as well. If $m = 0$ or $n = 0$, we get that either $r=1$ or $s=1$, which is a contradiction to the assumptions, so $m = n = 1$. \\

        Since $n = m = 1$, our sequences from above are $(r_i) = (r), (s_i) = (s)$, so by construction $\varphi(r), \varphi(s)$ are in one of the factors $H_1,H_2$. We shall consider the case $\varphi(r),\varphi(s) \in H_1$ and conclude that $\exists j \in \Z : \varphi(r),\varphi(s) \sim g_1^j$ in $\pi_1(\Sigma_2)$. The other case, where $\varphi(r),\varphi(s) \in H_2$ and we conclude that $\exists j \in \Z : \varphi(r),\varphi(s) \sim g_3^j$ in $\pi_1(\Sigma_2)$), is analogous.
        
        Since $\varphi(r),\varphi(s) \in H_1$, note that $r,s \in \langle a,b \rangle$. Recall, it was assumed that $r \not\sim s$ and $\varphi(r) \sim \varphi(s)$, and want to show that $\exists j \in \Z: \varphi(r), \varphi(s) \sim g_1^j$.
        Denote $G = \langle a,b \rangle$, and $\psi = \varphi \restriction_G : G \to H_1$. Note that the conjugacy classes of $G$ are:
        \begin{enumerate}
            \item $[1]_G$;
            \item $[\Pi_i a^{k_i}b^{l_i}]_G$ for some $0 \neq k_i,l_i \in \Z$;
            \item $[a^k]_G$ for some $0 \neq k \in \Z$;
            \item $[b^l]_G$ for some $0 \neq l \in \Z$,
        \end{enumerate}
        where some of the conjugacy classes listed in case 2 above are not distinct (i.e. $[aba^2b^2]_G = [a^2b^2ab]_G$); this will not matter to our argument. Define $\tilde{\psi}$, a function from the set of conjugacy classes of $G$ to the set of conjugacy classes of $H_1$, by
        \begin{myequation}
            \tilde{\psi}([x]_G) = [\psi(x)]_{H_1} .
        \end{myequation}
        Our assumptions can be rewritten as $[r]_G \neq [s]_G, \tilde{\psi}([r]_G) = \tilde{\psi}([s]_G)$, and we want to show $\tilde{\psi}([r]_G) = [g_1^j]_{H_1}$ for some $0 \neq j \in \Z$. Calculate $\tilde{\psi}$ for all the conjugacy classes of $G$ as listed above:
        \begin{enumerate}
            \item $\tilde{\psi}([1]_G) = [1]_{H_1}$;
            \item $\tilde{\psi}([\Pi_i a^{k_i}b^{l_i}]_G) = [\Pi_i g_1^{k_i} g_2 g_1^{l_i} g_2^{-1}]_{H_1}$ ($\forall i: 0 \neq k_i,l_i \in \Z$);
            \item $\tilde{\psi}([a^k]_G) = [g_1^k]_{H_1}$ ($0 \neq k \in \Z$);
            \item $\tilde{\psi}([b^l]_G) = [g_2 g_1^l g_2^{-1}]_{H_1} = [g_1^l]_{H_1}$ ($0 \neq l \in \Z$).
        \end{enumerate}
        
        Therefore:
        \begin{enumerate}
            \item If $\tilde{\psi}([r]_G) = \tilde{\psi}([s]_G) = [1]_{H_1}$, then $\psi(r) = \psi(s) = 1$ and then $r = s = 1$ (since $\psi$ is a monomorphism), in contradiction.
            \item If $\tilde{\psi}([r]_G) = \tilde{\psi}([s]_G) = [\Pi_i g_1^{k_i} g_2 g_1^{l_i} g_2^{-1}]_{H_1}$, then $[r]_G = [s]_G = [\Pi_i a^{k_i}b^{l_i}]_G$, in contradiction.
        \end{enumerate}
        
        The only cases which are left are $\tilde{\psi}([r]_G) = \tilde{\psi}([s]_G) = [g_1^j]$ for some $0 \neq j \in \Z$.
    \end{proof}
    
    \begin{remark}
    In exactly the same way, one can prove the following:
    
    Let $n \geq 2$, and define
    \begin{myequation}
        \varphi: F_{2n-1} = \langle a_1,...,a_{2n-1} \rangle \to \pi_1(\Sigma_n) = \langle g_1,...,g_{2n} | [g_1,g_2][g_3,g_4]...[g_{2n-1},g_{2n}] \rangle ,\\
        a_{2i-1} \mapsto g_{2i-1} \ (\forall 1 \leq i \leq n) ,\\
        a_{2i} \mapsto g_{2i} g_{2i-1} g_{2i}^{-1} \ (\forall 1 \leq i < n) .
    \end{myequation}
    Let $r,s \in F_{2n-1}$, and assume that $\varphi(r) \sim \varphi(s)$ in $\pi_1(\Sigma_n)$ and $r \not\sim s$ in $F_{2n-1}$. Then $\exists 1 \leq i \leq n, 0 \neq j \in \Z$ such that $\varphi(r),\varphi(s) \sim g_{2i-1}^j$ in $\pi_1(\Sigma_n)$. \\
    \end{remark}
    
    The lemma is a corollary of Claim~\ref{cl:1}.
    
    \begin{proof}[Proof of Lemma~\ref{lemma:2}]
        Define
        \begin{myequation}
            A : F_3 \to F_3 ,\\
            a \mapsto ac ,\\
            b \mapsto b^{-1}c ,\\
            c \mapsto c .
        \end{myequation}
        This is an automorphism of $F_3$. Note that $i_* = \varphi \circ A$ (with $\varphi$ as defined above). We want to use Claim~\ref{cl:1} with $r = A(\beta), s = A(\delta)$:
        \begin{myequation}
            A(\beta) = (ac)^{m_1} (b^{-1}c)^{n_1} ... (ac)^{m_p} (b^{-1}c)^{n_p} ,\\
            A(\delta) = (ac)^{k_1} c^{-u_1} (b^{-1}c)^{l_1} c^{v_1} ... (ac)^{k_p} c^{-u_p} (b^{-1}c)^{l_p} c^{v_p} .
        \end{myequation}
        
        Assume by contradiction that $\delta \not\sim \beta$ in $F_3$. Then $A(\beta) \not\sim A(\delta)$ in $F_3$, since $A$ is an automorphism of $F_3$, and then by Claim~\ref{cl:1}, $\exists 0 \neq j \in \Z$ such that one of the following holds:
        \begin{itemize}
            \item $[A(\beta)]_{F_3} = [c^j]_{F_3}$;
            \item $[A(\beta)]_{F_3} = [(ab^{-1}c)^j]_{F_3}$;
            \item $[A(\beta)]_{F_3} = [b^j]_{F_3}$;
            \item $[A(\beta)]_{F_3} = [a^j]_{F_3}$.
        \end{itemize}
        
        To see that, for example, $[A(\beta)]_{F_3} = [c^j]_{F_3}$ leads to a contradiction, consider the projection
        \begin{myequation}
            p_{a,b} : \langle a,b,c \rangle \to \langle a,b \rangle ,\\
            a \mapsto a, b \mapsto b, c \mapsto 1 .
        \end{myequation}
        If $[A(\beta)]_{F_3} = [c^j]_{F_3}$, then $[(a)^{m_1} (b^{-1})^{n_1} ... (a)^{m_p} (b^{-1})^{n_p}]_{\langle a,b \rangle} = [p_{a,b}(A(\beta))]_{\langle a,b \rangle} = [p_{a,b}(c^j)]_{\langle a,b \rangle} = [1]_{\langle a,b \rangle}$. This implies that at least one of the $m_i,n_i$ is $0$, else no cancellation can occur in $(a)^{m_1} (b^{-1})^{n_1} ... (a)^{m_p} (b^{-1})^{n_p}$. This is a contradiction to the assumption $\forall i: 0 \neq m_i,n_i$.
        
        The other three cases are dealt with similarly: the case $[A(\beta)]_{F_3} = [(ab^{-1}c)^j]_{F_3}$ using the projection
        \begin{myequation}
            p_{ab^{-1}c,c} : \langle a,b,c \rangle \to \langle a,b \rangle ,\\
            ab^{-1}c \mapsto a, c \mapsto b, a \mapsto 1 .
        \end{myequation}
        (note this is well defined since $\{ab^{-1}c, a, c\}$ is a basis of $F_3$), and the cases $[A(\beta)]_{F_3} = [b^j]_{F_3}$, $[A(\beta)]_{F_3} = [a^j]_{F_3}$ using $p_{a,b}$ defined above.
        
        Since all cases lead to a contradiction, we conclude that $\delta \sim \beta$ in $F_3$, as desired. 
    \end{proof}
    
    Recall $\pi_1(\Sigma_2) = \langle H_1*H_2 , A=B , \phi \rangle$ and recall the definition of $len: \pi_1(\Sigma_2) \to \Z_{\geq 0}$. We will also make use of the following claim:
    \begin{claim}
    \label{cl:length}
        Let $r \in \N, k_i \in \Z$ for $i \in \{1,...,2r\}$ with $|k_i| \geq 2$ for all $i$, and consider $\pi_1(\Sigma_2) = \langle g_1,g_2,g_3,g_4 | [g_1,g_2][g_3,g_4] \rangle$. Then
        \[
            len \left( \Pi_{i=1}^r (g_1 g_3)^{k_{2i-1}} (g_2 g_1^{-1} g_2^{-1}g_3)^{k_{2i}} \right) > 1 .
        \]
    \end{claim}
    
    \begin{proof}
        Denote $\epsilon_i = sign(k_i) = \frac{k_i}{|k_i|}$. Consider
        \begin{myequation}
            p_{13}: \pi_1(\Sigma_2) \to F_2 = \langle h_1, h_3 \rangle ,\\
            g_1 \mapsto h_1 ,\\
            g_3 \mapsto h_3 ,\\
            g_2, g_4 \mapsto 1 .
        \end{myequation}
        
        Denote $w = \Pi_{i=1}^r (g_1 g_3)^{k_{2i-1}} (g_2 g_1^{-1} g_2^{-1} g_3)^{k_{2i}}$. Assume by contradiction $len(w) \leq 1$. Note that $p_{13}$ preserves factors, i.e. $p_{13}(H_1) = \langle h_1 \rangle, p_{13}(H_2) = \langle h_3 \rangle$. If $len(w) = 0$, then $w = 1$, so $p_{13}(w) = 1$, and then $len(p_{13}(w)) = 0$. Else, $len(w) = 1$, so $w \sim c_1$ with $c_1 \in H_1 < \pi_1(\Sigma_2)$ or $c_1 \in H_2 < \pi_1(\Sigma_2)$. WLOG assume $c_1 \in H_1$. Then $p_{13}(w) \sim p_{13}(c_1) \in p_{13}(H_1) = \langle h_1 \rangle$, and then $len(p_{13}(w)) = 1$. One reaches the same conclusion in both cases: $len(p_{13}(w)) \leq 1$ (note that this is $len$ in $F_2$, which is a free product of $\Z = \langle h_1 \rangle$ with $\Z = \langle h_3 \rangle$ and so is trivially a free product with amalgamation). As a remark, this procedure can be applied for any factor-preserving homomorphism $\psi$ to get $len(g) = len(\psi(g))$.
        
        Consider $p_{13}(w)$:
        \[
            p_{13}(w) = \Pi_{i=1}^r (h_1 h_3)^{k_{2i-1}} (h_1^{-1} h_3)^{k_{2i}} .
        \]
        This element is made of blocks: $(h_1 h_3)^{k_j}$ or $(h_1^{-1} h_3)^{k_j}$. Each of these blocks is reduced, so the only cancellations in the form of $p_{13}(w)$ given above can happen between two blocks. These are the options for cancellations between two blocks:

        \begin{itemize}
        \item If the first block is $(h_1 h_3)^{k_i}$ and the second is $(h_1^{-1} h_3)^{k_{i+1}}$, one of the following holds:
        \begin{myequation}
                \epsilon_i = \epsilon_{i+1} = 1: (h_1 h_3)^{|k_i|} (h_1^{-1} h_3)^{|k_{i+1}|} \rightsquigarrow \text{no cancellations} ; \\
                \epsilon_i = -\epsilon_{i+1} = 1: (h_1 h_3)^{|k_i|} (h_1^{-1} h_3)^{-|k_{i+1}|} \rightsquigarrow (h_1 h_3)^{|k_i|-1} h_1 h_1 (h_1^{-1} h_3)^{-|k_{i+1}|+1} \ (*) ; \\
                - \epsilon_i = \epsilon_{i+1} = 1: (h_1 h_3)^{-|k_i|} (h_1^{-1} h_3)^{|k_{i+1}|} \rightsquigarrow \text{no cancellations} ; \\
                \epsilon_i = \epsilon_{i+1} = -1: (h_1 h_3)^{-|k_i|} (h_1^{-1} h_3)^{-|k_{i+1}|} \rightsquigarrow \text{no cancellations}.
            \end{myequation}
            
            Every cancellation of type $(*)$ results in a reduced block starting and ending with $h_1$. Before it comes either $h_3$ or $h_1$ (since $(h_1^{-1} h_3)^{k_{i-1}}$ ends with one of these symbols), and after it comes either $h_1$ or $h_3^{-1}$ (for similar reasons). Thus after a cancellation of this type, none of the two ends of the resulting reduced block can be cancelled any further.
            
        \item If the first block is $(h_1^{-1} h_3)^{k_i}$ and the second is $(h_1 h_3)^{k_{i+1}}$, one of the following holds:
        \begin{myequation}
                \epsilon_i = \epsilon_{i+1} = 1: (h_1^{-1} h_3)^{|k_i|} (h_1 h_3)^{|k_{i+1}|} \rightsquigarrow \text{no cancellations} ; \\
                \epsilon_i = -\epsilon_{i+1} = 1: (h_1^{-1} h_3)^{|k_i|} (h_1 h_3)^{-|k_{i+1}|} \rightsquigarrow (h_1^{-1} h_3)^{|k_i|-1} h_1^{-1} h_1^{-1} (h_1 h_3)^{-|k_{i+1}|+1} \ (**) ; \\
                - \epsilon_i = \epsilon_{i+1} = 1: (h_1^{-1} h_3)^{-|k_i|} (h_1 h_3)^{|k_{i+1}|} \rightsquigarrow \text{no cancellations} ; \\
                \epsilon_i = \epsilon_{i+1} = -1: (h_1^{-1} h_3)^{-|k_i|} (h_1 h_3)^{-|k_{i+1}|} \rightsquigarrow \text{no cancellations}.
            \end{myequation}
            
            Every cancellation of type $(**)$ results in a reduced block starting and ending with $h_1^{-1}$. Before it comes either $h_3$ or $h_1^{-1}$ and after it comes either $h_1^{-1}$ or $h_3^{-1}$ (for reasons similar to above). Thus after a cancellation of this type, none of the two ends of the resulting reduced block can be cancelled any further.
        \end{itemize}
        
        Therefore, after at most $2r-1$ cancellations one reaches a reduced word. If it is not cyclically reduced, conjugate by the first block and perform one more reduction to get a cyclically reduced word $v$, conjugate to $p_{13}(w)$. Since every reduced block $(*)$ or $(**)$ above has length at least 2, and these are inserted without cancellations to our word $v$, we have $len(p_{13}(w)) = len(v) \geq 2 > 1$. This is a contradiction, so the initial assumption $len(w) \leq 1$ is false.
    \end{proof}

\section{Proof of the theorems}\label{sec:3}
    \subsection{Additional definitions}
        In the proofs of the results, we shall use the following definitions.
        
        \begin{definition}
            A continuous map between two manifolds $f: M \to N$ induces a map $\tau_f: \pi_0(\mathscr{L}M) \to \pi_0(\mathscr{L}N)$ by $\tau_f([\gamma]) = [f \circ \gamma]$, for $\gamma: S^1 \to M$. \\
            
            For any manifold $M$ and point $x \in M$, define the map
            \begin{myequation}
                \eta_{M,x} : \pi_1(M,x) \to \pi_0(\mathscr{L}M) , \\
                [\tilde{\gamma}]_{\pi_1(M,x)} \mapsto [\tilde{\gamma}]_{\pi_0(\mathscr{L}M)}.
            \end{myequation}
            
            where $\tilde{\gamma}: S^1 \to M$ is a loop.
        \end{definition}
        Note that the following diagram commutes for any continuous map $f: M \to N$:
        \begin{center}\begin{tikzcd}
            \pi_1(M,x) \arrow[r, "f_*"] \arrow[d, "\eta_{M,x}"] & \pi_1(N,f(x)) \arrow[d, "\eta_{N,f(x)}"] \\
            \pi_0(\mathscr{L}M) \arrow[r, "\tau_f"]         & \pi_0(\mathscr{L}N)
        \end{tikzcd}\end{center}
        Note additionally that for $\alpha, \beta \in \pi_1(M,x)$, $\eta_{M,x}(\alpha) = \eta_{M,x}(\beta)$ if and only if $\alpha$ and $\beta$ are conjugate in $\pi_1(M,x)$.
        
        \begin{definition}
            A word $w \in F_2 = \langle V,H \rangle$ is called \textit{balanced} if it is of the form $w = V^{N_1} H^{M_1} ... V^{N_r} H^{M_r}$ for some $r \in \N$, $N_j,M_j \in \Z \setminus \{0\}$.
        \end{definition}
        Note that any balanced word is cyclically reduced.
        
        \begin{definition}
            Let $\varphi: M \to M$ be a diffeomorphism. A fixed point $x$ of $\varphi$ is called non-degenerate if $d\varphi_x$ does not have 1 as an eigenvalue.
        \end{definition}
        
        The egg-beater construction uses dynamics on a space $C$ which is then embedded into a surface of genus 2. Thus we shall use the following definitions of pushing forward the dynamics along the embedding.
        
        Let $X,Y$ be compact topological spaces, and $i: X \hookrightarrow Y$ a continuous embedding. Let $f: X \to \R$ be a continuous map on $X$, and assume the following condition holds:
        \begin{equation}
        \label{eq:condition}
            \text{For any path-component $C$ of $Y \setminus i(X)$, $f \restriction_{i^{-1}(\partial C)}$ is constant.}
        \end{equation}
            
        Let $C_y$ be the path-component of $Y$ that contains $y \in Y$, and denote $D_i = \bigcup_{y \in \Ima(i)} C_y \subseteq Y$. For all $y \in D_i$, denote by $\gamma_{i,y}: [0,1] \to C_y$ a continuous path with $\gamma_{i,y}(0) = y$, $\gamma_{i,y}(1) \in \Ima(i)$, and such that if $\gamma_{i,y}(t) \in \Ima(i)$ for some $t \in [0,1]$, then $\gamma_{i,y} \restriction_{[t,1]}$ is constant. Note that if $y \in \Ima(i)$, then $\gamma_{i,y} \equiv y$.
        
        Denote the following, not necessarily continuous, map:
        \begin{myequation}
            b_i: D_i \to \Ima(i) ,\\
            y \mapsto \gamma_{i,y}(1) .
        \end{myequation}
        Define the following map, the \textit{pushforward} of $f$ through $i$:
        \begin{myequation}
            i_* f: D_i \to \R ,\\
            y \mapsto f \circ i^{-1} \circ b_i(y) .
        \end{myequation}
        By Condition~\ref{eq:condition}, this is a continuous map $D_i \to \R$ which is constant on $Y \setminus i(X)$, and doesn't depend on the choice of the $\gamma_{i,y}$s. Note also that if $f$ is smooth and constant in a neighborhood of $i^{-1}(\partial \ i(X))$ then $i_* f$ is also smooth. 
        
        Assume additionally that $X,Y$ are symplectic manifolds, $i$ is a symplectomorphism, and $f: S^1 \times X \to \R$ is a Hamiltonian function. Then $f$ induces the time-one-map of its flow; let it be denoted $F: X \to X$. The \textit{pushforward} of $F$ through $i$ is denoted $i_* F: Y \to Y$, and is the time-one-map of the flow induced by $i_* f$.
        
    \subsection{Review of the original proofs (genus $\geq 4$)}
    \label{sec:3.1}
        \subsubsection{Outline of the original proofs and dependencies of claims}
            First we shall outline the proofs of the original theorems, 1.3 from~\cite{PS14} and 1.1 from~\cite{10}. For concreteness, their statements are given here.
            
            \begin{theorem}[Theorem 1.3 from~\cite{PS14}]
            \label{thm:orig2}
                Let $\Sigma_4$ be a closed oriented surface of genus $\geq 4$, equipped with an area form $\sigma_4$, and $k \geq 2$ an integer. Then $powers_k(\Sigma_4,\sigma_4) = \infty$.
            \end{theorem}
            
            \begin{theorem}[Theorem 1.1 from~\cite{10}]
            \label{thm:orig3}
                Let $\Sigma_4$ be a closed oriented surface of genus $\geq 4$, equipped with an area form $\sigma_4$. Then for any non-principal ultrafilter $\mathcal{U}$ on $2^\N$, there exists a monomorphism $F_2 \hookrightarrow Cone_\mathcal{U}(Ham(\Sigma_4),d_H)$.
            \end{theorem}
            
            Both their proofs involve several intermediate results. For clarity, Figure~\ref{fig:orig_depends} lists the dependencies between these results. The proofs are based on a specific construction of a manifold $C$, an embedding $i: C \hookrightarrow \Sigma_4$, and some dynamics on $C$, which induce dynamics on $\Sigma_4$. These constructions are all sketched in Subsection~\ref{sec:outline}, and given in detail further in this subsection. Propositions~\ref{prop:main},~\ref{prop:main2} (which are Propositions 5.11 of~\cite{10} and 5.1 of~\cite{PS14} respectively) are claims on the dynamics on $\Sigma_4$, and Claims~\ref{cl:2},~\ref{cl:2.2} are the respective claims on the dynamics on $C$.
            
            The deductions marked $**$ use "hard" Floer homology and persistence modules, and do not depend on the genus of $\Sigma_4$, therefore they can be taken as is to the case where the surface is of genus 2,3. The deductions marked $*$ are a consequence of the fact that $i$ is constructed to be \textit{incompressible}, i.e. it induces injections $\pi_1(C) \hookrightarrow \pi_1(\Sigma_4)$ and $\pi_0(\mathscr{L}C) \hookrightarrow \pi_0(\mathscr{L}\Sigma_4)$. This will not hold in the case where $\Sigma$ is of genus 2,3 (see Subsection~\ref{sec:3.2} below), and this is exactly where Lemma~\ref{lemma:2} comes into play. Claims~\ref{cl:2},~\ref{cl:2.2} are proved directly by careful analysis of the dynamics on $C$.
            
            This subsection details the construction $i: C \hookrightarrow \Sigma_4$, the dynamics on $C$ and $\Sigma_4$, states all the above claims and propositions, and outlines their proofs. For full details, see~\cite{PS14} and~\cite{10}; we describe their work bottom-up, with respect to the directions of Figure~\ref{fig:orig_depends}. \\
            
            \begin{figure}
                \centering
                \begin{tikzcd}
                    Claim~\ref{cl:2} \arrow[d,rightsquigarrow,"*"] & Claim~\ref{cl:2.2} \arrow[d,rightsquigarrow,"*"] & \bigg] in \ C \\
                    Proposition~\ref{prop:main} \arrow[d,rightsquigarrow,"**"] & Proposition~\ref{prop:main2} \arrow[dd,rightsquigarrow,"**"] & \bigg] in \ \Sigma_4 \\
                    Theorem~\ref{thm:4} \arrow[d,rightsquigarrow] \\
                    Theorem~\ref{thm:orig3} & Theorem~\ref{thm:orig2}
                \end{tikzcd}
                \caption{Dependencies between original intermediate results.}
                \label{fig:orig_depends}
            \end{figure}
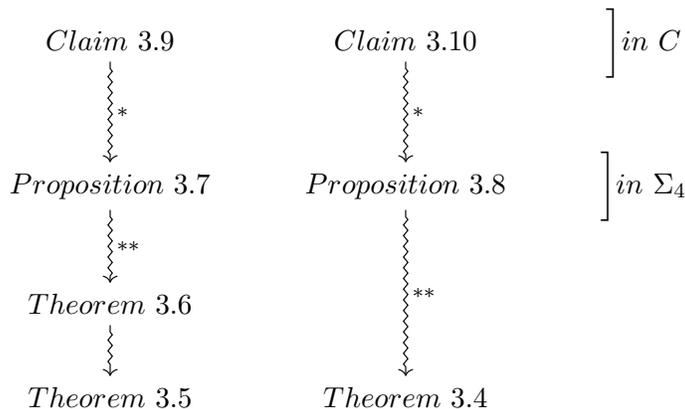
            
            The manifold $C$ mentioned above is the union of two annuli $[-1,1] \times \R/L\Z$ for some $L>4$, which intersect in two squares (see Figure~\ref{fig:qs} below). Special dynamics, called eggbeater dynamics, are defined on $C$ (see definition later in this subsection). In order to get results for a symplectic surface $M$, we need an embedding $i: C \hookrightarrow M$, which will induce eggbeater dynamics on $M$.
            
            Since $M$ is symplectic, it is orientable, and this leaves little choice for its homeomorphism type - the genus of $M$ determines it. Topologically, one can always think of $i$ as first embedding $C$ into the sphere, and then adding some handles. The results here use Floer homology, where basic objects of interest are homotopy classes of free loops in the manifold. Thus we want the embedding $i$ to have the property that homotopy classes of different orbits of the dynamics on $C$ will be pushed by $i$ to different homotopy classes in $M$, so that we will be able to distinguish between different orbits with Floer homology tools.
            
            Since the only choice in the embedding is how many handles to add and in which components of $M \setminus i(C)$ to attach them, we want every such component to have at least one end of a handle - otherwise, it will be contractible. There are four such components (consider Figure~\ref{fig:qs}), so for this method to work we need $M$ to be of genus at least 2. The original construction, found in~\cite{PS14} and presented later in this subsection, uses a surface of genus at least 4, and adds (at least) one handle in every such component. This produces results for surfaces of genus $\geq 4$, and has the added benefit that it makes $i$ incomressible - so that different free homotopy classes of loops in $C$ do not merge under $i$. However, this is not efficient if we want to minimize the genus. A slightly more efficient construction is found in Section~\ref{sec:3.2}, and produces results for surfaces of genus 2,3. \\

            Both of the original proofs of Theorems~\ref{thm:orig2},~\ref{thm:orig3} use a construction of a sequence of homomorphisms $\Phi_k: F_2 \to Ham(\Sigma_4)$ (indexed by $k \in \N$). Every such Hamiltonian diffeomorphism $\Phi_k(w)$ is generated by a specific Hamiltonian denoted $H_{k,w}$, and the Hamiltonian isotopy generated by $H_{k,w}$ is denoted $\phi_{k,w}(t)$. These are all specified exactly later in this subsection. Note that these homomorphisms depend on the surface $\Sigma_4$ - the proofs for the generalized theorems (in Subsection~\ref{sec:3.2}) will use similiar but different homomorphisms.
            
            This sequence $\Phi_k$ induces a homomorphism
            \begin{myequation}
                F_2 \to Cone_\mathcal{U}(Ham(\Sigma), d_H) , \\
                w \mapsto [(\Phi_k(w))_{k=1}^\infty] .
            \end{myequation}
            
            To show this is a monomorphism, and so prove Theorem~\ref{thm:orig3}, for any $1 \neq w \in F_2$ we need to show that $\lim_\mathcal{U} \frac{d_H(\Phi_k(w), id_\Sigma)}{k} = \lim_\mathcal{U} \frac{|| \Phi_k(w) ||_H}{k} > 0$. This is done in Theorem~\ref{thm:4}:
            
            \begin{theorem}[Theorem 2.1 from~\cite{10}]
            \label{thm:4}
                Let $1 \neq w \in F_2 = \langle H,V \rangle$. Then there exist constants $C = C(w) > 0, k_0 = k_0(w) \in \N$ such that for any $k > k_0$:
                \[
                    || \Phi_k(w) ||_H \geq C \cdot k.
                \]
            \end{theorem}
            
            \begin{corollary}
                Theorem~\ref{thm:4} implies Theorem~\ref{thm:orig3}.
            \end{corollary}
            
            In addition to the above sequence $\Phi_k$, a collection of free homotopy classes $\alpha_{k,w} \in \pi_0(\mathscr{L}C)$, indexed by $k \in \N, w \in F_2$, and another collection of free homotopy classes $\alpha_k^\prime \in \pi_0(\mathscr{L}C)$, indexed by $k \in K$, are specified, where $K \subset \N$ is some unbounded subset to be specified later in this subsection.
            
            To prove Theorems~\ref{thm:orig2},~\ref{thm:4}, the following propositions are used:
            
            \begin{prop}[Proposition 5.11 from~\cite{10}]
            \label{prop:main}
                Let $w = \Pi_j V^{N_j} H^{M_j} \in F_2$ be a balanced word. For large enough $k \in \N$, there are $2^{2r}$ non-degenerate fixed points of $\Phi_k(w)$ whose orbits have free homotopy class $\tau_i(\alpha_{k,w})$ (i.e. $[t \mapsto \phi_{k,w}(t)(z_0)]_{\pi_0(\mathscr{L}\Sigma_4)} = \tau_i(\alpha_{k,w})$ for $z_0$ a non-degenerate fixed point of $\Phi_k(w)$), and these fixed points are indexed by $\vec{\epsilon} = (\epsilon_0,...,\epsilon_{2r-1}) \in \{\pm 1\}^{2r}$, with the fixed point associated to sign vector $\vec{\epsilon}$ denoted $z(\vec{\epsilon})$. The action and Conley-Zehnder index of the point $z(\vec\epsilon)$ are:
                \begin{equation}
                \label{eq:2}
                    \mathcal{A}(z(\vec{\epsilon})) = Lk \sum_{j=1}^r \left( \epsilon_{2j-2} N_j \left(1-\frac{1}{2|N_j|}\right)^2 + \epsilon_{2j-1} M_j \left(1-\frac{1}{2|M_j|}\right)^2 \right) + O(1) ,
                \end{equation}
                \begin{equation}
                \label{eq:3}
                    \mu_{CZ}(z(\vec\epsilon)) = 1 + \frac12 \sum_{j=1}^r \left( \epsilon_{2j-2} sign(N_j) + \epsilon_{2j-1} sign(M_j) \right) ,
                \end{equation}
                where the action of a fixed point is understood to be that of its orbit under $\phi_{k,w}$, the action and Conley-Zehnder index are with respect to the Hamiltonian $H_{k,w}$, and where the $O$ notation in Equation~\ref{eq:2} is as $k \to \infty$.
            \end{prop}
            
            \begin{prop}[Proposition 5.1 from~\cite{PS14}]
            \label{prop:main2}
                Let $w = (VH)^r \in F_2$ for some $r \in \N$. For large enough $k \in K$, there are $2^{2r}$ non-degenerate fixed points of $\Phi_k(w)$ whose orbits have free homotopy class $\tau_i(\alpha_k^\prime)$, and fixed points in different orbits have action gaps that grow linearly with $k$: for such fixed points in different orbits $y, z$:
                \[
                    | \mathcal{A}(y) - \mathcal{A}(z) | \geq c \cdot k + O(1) ,
                \]
                as $k \to \infty$, for some global constant $c > 0$.
            \end{prop}
            
            Note that both propositions depend on the definition of the embedding $i$, defined later in this subsection.
            
            Given Proposition~\ref{prop:main2}, Theorem~\ref{thm:orig2} is proven in Section 5.1 of~\cite{PS14}, and given Proposition~\ref{prop:main}, Theorem~\ref{thm:4} is proven in Section 5.4 of~\cite{10} (the case where $w$ is not conjugate to a power of $V$ or $H$ requires Proposition~\ref{prop:main}; the case where $w$ is conjugate to a power of $V$ or $H$ does not use it). The remainder of this subsection is devoted to outlining the construction of $\Phi_k, \phi_{k,w}, \alpha_{k,w}, \alpha_k^\prime$ and the proof of Propositions~\ref{prop:main},~\ref{prop:main2}, as proven in~\cite{PS14},~\cite{10}. \\
        
        \subsubsection{The geometric construction}\label{sec:geometric}
            Let $\Sigma_4$ be a surface of genus $\geq 4$. Consider the cylinder $C_*=[-1,1]\times\mathbb{R}/L\mathbb{Z}$, for $L > 4$, with coordinates $x,y$ and the standard symplectic form $dx\wedge dy$. Denote by $C_V,C_H$ two copies of $C_*$. Consider the squares
            \[
                S_0^\prime = [-1,1]\times[-1,1]/L\Z, S_1^\prime = [-1,1]\times[L/2-1,L/2+1]/L\Z
            \]
            in $C_*$. They give four squares $S_{V,0},S_{V,1}\subset C_V$ and $S_{H,0},S_{H,1}\subset C_H$. Define the symplectomorphism $VH_{0,1}:S_{V,0}\bigsqcup S_{V,1} \to S_{H,0} \bigsqcup S_{H,1}$ given by $VH\bigsqcup VH^\prime$, where
            \begin{myequation}
                VH:S_{V,0}\to S_{H,0} : (x,[y]) \mapsto (-y,[x]) ,\\
                VH^\prime : S_{V,1}\to S_{H,1} : (x,[y]) \mapsto (y-L/2, [-x+L/2]) .
            \end{myequation}
            Define
            \[
                C = C_V \bigcup_{VH_{0,1}} C_H .
            \]
            This is a symplectic manifold with symplectic form $\omega_0 = dx \wedge dy$ on every copy of the cylinder $C_*$. Denote by $C_V: C_* \hookrightarrow C, \ C_H: C_* \hookrightarrow C$ the two injections induced by the above union.
            
            Denote by $S_0,S_1 \subset C$ the identification of the squares $S_{V,0},S_{H,0}$ and $S_{V,1},S_{H,1}$. In fact, $S_0 \cup S_1 = C_V \cap C_H$. Fix two points $s_0 \in S_0, s_1 \in S_1$. Define 4 paths: two paths $q_1,q_3$ from $s_0$ to $s_1$, and two paths $q_2,q_4$ from $s_1$ to $s_0$ (see Figure~\ref{fig:qs}); $q_1, q_2$ are paths on $C_V$, and $q_3,q_4$ are paths on $C_H$.
    
            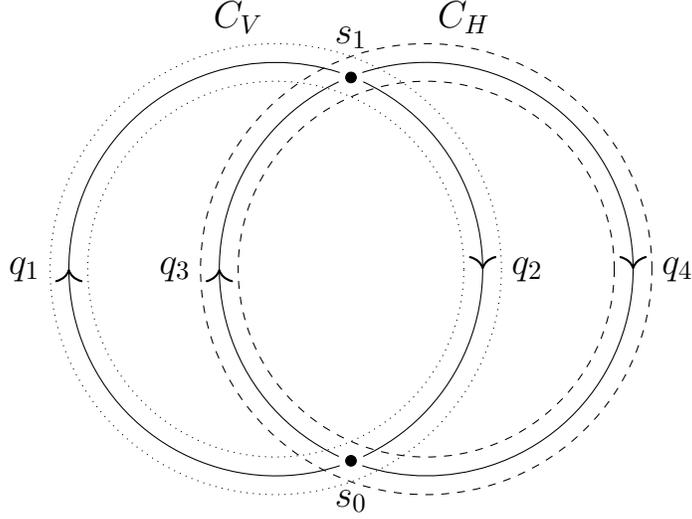
\begin{figure}
                \centering
                \begin{tikzpicture}
                    \begin{scope}[decoration={markings,mark=at position 0.5 with {\arrow[scale=2]{>}}}]
                        \draw[dotted] (-1,0) circle (3);
                        \draw[dotted] (-1,0) circle (2.5);
                        \draw (-1.5,3) node[anchor=south] {\Large $C_V$};
                        
                        \draw[dashed] (1,0) circle (3);
                        \draw[dashed] (1,0) circle (2.5);
                        \draw (1.5,3) node[anchor=south] {\Large $C_H$};
                        
                        \filldraw[black] (0,-2.55) circle (2pt);
                        \draw (0,-2.8) node[anchor=north] {\Large $s_0$};
                        \filldraw[black] (0,2.55) circle (2pt);
                        \draw (0,2.8) node[anchor=south] {\Large $s_1$};
                        
                        \draw[postaction={decorate}] (-1,0) ++(288:2.75) arc (288:72:2.75);
                        \draw[postaction={decorate}] (-1,0) ++(65:2.75) arc (65:-65:2.75);
                        \draw[postaction={decorate}] (1,0) ++(468:2.75) arc (468:252:2.75);
                        \draw[postaction={decorate}] (1,0) ++(245:2.75) arc (245:115:2.75);
                        
                        \draw (-4,0) node[anchor=east] {\Large $q_1$};
                        \draw (2,0) node[anchor=west] {\Large $q_2$};
                        \draw (-2,0) node[anchor=east] {\Large $q_3$};
                        \draw (4,0) node[anchor=west] {\Large $q_4$};
                    \end{scope}
                \end{tikzpicture}
                \caption{The paths $q_1, q_2, q_3, q_4$ in $C$.}
                \label{fig:qs}
            \end{figure}
    
            Note that $\pi_1(C,s_0) \simeq F_3$, the free group on 3 generators. The 3 generators $a,b,c$ of $\pi_1(C,s_0)$ are taken to be:
            \begin{myequation}
            a = [q_1 \# q_2]_{\pi_1(C,s_0)} , \\
            b = [q_3 \# q_4]_{\pi_1(C,s_0)} , \\
            c = [q_3 \# q_2]_{\pi_1(C,s_0)} ,
            \end{myequation}
            where the $\#$ sign is used for path concatenation: $q\#q^\prime$ is the concatenation of paths $q$ and then $q^\prime$, if $q(1) = q^\prime(0)$.
            
            Consider the function $u_0: [-1,1] \to \R$, $u_0(s) = 1 - |s|$. Take an even, non-negative, sufficiently $C^0$-close smoothing $u$ to $u_0$ such that $u$ is supported away from $\{\pm 1\}$, and both $u-u_0$ and $\int_{-1}^t (u(s)-u_0(s)) ds$ are supported in a sufficiently small neighborhood of $\{\pm 1, 0\}$. For $k \in \N$, define
            \begin{myequation}
                f = f_k: C_* \to C_* , \\
                f(x,[y]) = (x,[y + k L \cdot u(x)]) .
            \end{myequation}
            
            This mapping is a Hamiltonian diffeomorphism on $C_*$ with Hamiltonian
            \[
                h_k(x,[y]) = -\frac{k}{2} + k \int_{-1}^x u(s) ds .
            \]
            
            Denote $f_{k,V} = (c_V)_* f_k$, $f_{k,H} = (c_H)_* f_k$. Recall that these are two Hamiltonian diffeomorphisms on $C$, one supported on $C_V$ and the other on $C_H$, both supported away from $\partial C$. Define a homomorphism
            \begin{myequation}
                \Psi_k: F_2 = \langle V,H \rangle \to Ham_c(C,\omega_0) ,\\
                V \mapsto f_{k,V}, H \mapsto f_{k,H} ,
            \end{myequation}
            where $Ham_c(M,\omega)$ denotes the group of Hamiltonian diffeomorphisms whose generators have compact support of a symplectic manifold $(M,\omega)$.
            Note that the image of a word $w = V^{N_1} H^{M_1} ... V^{N_r} H^{M_r} \in F_2$ is $f_{k,H}^{M_r} \circ f_{k,V}^{N_r} \circ ... \circ f_{k,H}^{M_1} \circ f_{k,V}^{N_1}$. These images are called \textit{eggbeater maps}.
            
            Pushing forward the flow generated by $h_k$ to $C$, one gets flows on $C_V$ and $C_H$ whose time-one-maps are $f_{k,V}$ and $f_{k,H}$, and concatenating these in the order induced by $w$ (as in the construction of $\Psi_k(w)$) one gets a flow $\psi_{k,w}(t): C \to C$, whose time-one-map is $\Psi_k(w)$. Denote the Hamiltonian that generates this flow $G_{k,w}: S^1 \times C \to \R$. \\
            
            Consider a symplectic embedding $i: C \hookrightarrow \Sigma_4$ such that $i_*: \pi_1(C) \to \pi_1(\Sigma_4)$ and $\tau_i: \pi_0(\mathscr{L}C) \to \pi_0(\mathscr{L}\Sigma_4)$ are both injective, and such that each component of $\partial C$ separates $\Sigma_4$. For example, one can embed $C$ into $\R^2$, add at least one handle in every connected component of $\R^2 \setminus C$, and compactify by adding the point at infinity; see the construction in~\cite{PS14} and~\cite{10}. This is the point where Lemma~\ref{lemma:2} is used when generalizing the propositions to surfaces of genus 2,3 (see Subsection~\ref{sec:3.2}). Note that $i, f_{k,V}$ and $i, f_{k,H}$ satisfy Condition~\ref{eq:condition}. \\
            
            The homomorphisms $\Phi_k: F_2 \to Ham(\Sigma_4)$ will be defined by $\Phi_k(w) = i_* \Psi_k(w)$. The mapping $\Phi_k(w)$ is indeed a diffeomorphism, because of the assumptions on $u$, and is Hamiltonian, since it is generated by the Hamiltonian $H_{k,w} = i_* G_{k,w}$. \\
            
            Given $k \in \N, w = V^{N_1} H^{M_1} ... V^{N_r} H^{M_r} \in F_2$, set the free homotopy classes $\alpha_{k,w}$ to be
            \[
                \alpha_{k,w} = \eta_{C,s_0}(\beta_{k,w}) \in \pi_0(\mathscr{L}C) ,
            \]
            where
            \[
                \beta_{k,w} = \Pi_{j=1}^r a^{k \cdot sign(N_j)} b^{k \cdot sign(M_j)} \in \pi_1(C,s_0)
            \]
            and $sign: \Z \to \{0, \pm1\}$ is
            \begin{myequation}
                n \mapsto
                \left\{\begin{array}{ll}
                    \frac{n}{|n|}  & n \neq 0 \\
                    0              & n = 0
                \end{array}\right. .
            \end{myequation}
            
            Additionally, choose $\nu_j,\mu_j \in (0,1)$ for $j=1,...,r$ such that $\frac{\nu_j}{L},\frac{\mu_j}{L} \in \Q$ for all $j$, and the values $\sum_{j=0}^{r-1} (\epsilon_{2j+1}(1-\mu_{j+1})^2 - \epsilon_{2j+4}(1-\nu_{j+1})^2)$ are all distinct, for all sign vectors $\vec\epsilon = (\epsilon_0,...,\epsilon_{2r-1}) \in \{\pm 1\}^{2r}$. For these choices, the set $K \subset \N$ from Proposition~\ref{prop:main2} is taken to be $K = \{k \in \N | \forall j: \frac{\mu_j k}{L}, \frac{\nu_j k}{L} \in \Z \}$.
            
            For $k \in K$ set
            \[
                \alpha_k^\prime = \eta_{C,s_0}(\beta_k^\prime) \in \pi_0(\mathscr{L}C) ,
            \]
            where
            \[
                \beta_k^\prime = \Pi_{j=1}^r a^{\frac{\nu_j k}{L}} b^{\frac{\mu_j k}{L}} \in \pi_1(C,s_0) .
            \]
            
            From this point on, it is assumed that any index $k$ of $\alpha_k^\prime$ or $\beta_k^\prime$ is in the subset $K$, since otherwise $\alpha_k^\prime, \beta_k^\prime$ are not well-defined. This will also be explicitly stated. \\
            
            The outline of the proofs of Propositions~\ref{prop:main},~\ref{prop:main2} is as follows. Since any non-degenerate fixed point of $\Phi_k(w)$ must lie in $\Ima i$, and in addition $i: C \to \Sigma_4$, $i_*: \pi_1(C,s_0) \to \pi_1(\Sigma_4,i(s_0))$ and $\tau_i: \pi_0(\mathscr{L}C) \to \pi_0(\mathscr{L}\Sigma_4)$ are all injective, it is enough to prove versions of Propositions~\ref{prop:main},~\ref{prop:main2} on $C$, i.e. prove the following claims:
            \begin{claim}
            \label{cl:2}
                Let $w = \Pi_{j=1}^r V^{N_j} H^{M_j} \in F_2$ be balanced. For large enough $k \in \N$, there are $2^{2r}$ non-degenerate fixed points of $\Psi_k(w)$ whose orbits have free homotopy class $\alpha_{k,w}$, and these fixed points have actions (with respect to $G_{k,w}$) and Conley-Zehnder indices as in Equations~\ref{eq:2},\ref{eq:3}.
            \end{claim}
            
            \begin{claim}
            \label{cl:2.2}
                Let $w = (VH)^r \in F_2$ for some $r \in \N$. For large enough $k \in K$, there are $2^{2r}$ fixed points of $\Psi_k(w)$ whose orbits have free homotopy class $\alpha_k^\prime$, and such fixed points in different orbits have action gaps (with respect to $G_{k,w}$) that grow linearly with $k$ as $k \to \infty$.
            \end{claim}
            
            This is proven in two steps. First, one proves these claims with respect to a piecewise-linear version $\psi_{k,w}^\prime(t)$ of the isotopy $\psi_{k,w}(t)$. This is done in Subsection 5.1 of~\cite{10} and in Section 5 of~\cite{PS14}, by analysis of the dynamics on $C$. Then, one shows that the non-degenerate fixed points of $\Psi_k(w)$ are exactly those of $\psi_{k,w}^\prime(1)$, for $k$ large enough. This is done in Subsection 5.3 of~\cite{10}.
    
    \subsection{Proof of the new theorems (genera 2,3)}
    \label{sec:3.2}
        The new proofs are very similar to the original ones, with the following changes. The interesting dynamics, which we want to keep, are the eggbeater dynamics, but we will need to change our construction of $i, C, \psi_{k,w}$ a bit to accomodate for the fact that the genus of the surface is now 2 or 3.
        
        More concretely, let $\Sigma$ be a surface of genus 2 or 3 with area form $\sigma$. The following data will be defined:
        \begin{itemize}
            \item A symplectic surface $(D,\omega)$ that contains $C$ from the original construction (i.e. $e_2: C \hookrightarrow D$ is a symplectic embedding that induces an embedding $\tau_{e_2}: \pi_0(\mathscr{L}C) \hookrightarrow \pi_0(\mathscr{L}D)$),
            \item an injection $i_2: D \hookrightarrow \Sigma$, and 
            \item homomorphisms $\Xi_k: F_2 \to Ham_c(D, \omega)$ indexed by $k \in \N$, with $\Xi_k(w)$ the time-one-map of the flow $\xi_{k,w}(t): D \to D$, such that for all $k \in \N, w \in F_2$: $\xi_{k,w}(t) \restriction_C = \psi_{k,w}(t)$.
        \end{itemize}
        The Hamiltonian generating $\xi_{k,w}(t)$ will be denoted $F_{k,w}: S^1 \times D \to \R$. These data will specify the dynamics on $D$. In order to work on $\Sigma$, the dynamics will be pushed forward by $i_2$. This will result in a Hamiltonian $(i_2)_* F_{k,w}: S^1 \times \Sigma \to \R$, and the flow and time-one-map it generates, denoted $(i_2)_* \xi_{k,w}(t): \Sigma \to \Sigma$ and $(i_2)_* \Xi_k(w): \Sigma \to \Sigma$.
        
        The confused reader may consult the following diagram, which indicates the different surfaces mentioned till now and the injections between them, together with the Hamiltonians on them, the flows these Hamiltonians generate, and their time-one-maps:
        
        \begin{center}\begin{tikzcd}
            (C, G_{k,w}, \psi_{k,w}, \Psi_k(w)) \arrow[hookrightarrow]{d}{e_2} \arrow[hookrightarrow]{r}{i} & (\Sigma_4, H_{k,w} = i_* G_{k,w}, \phi_{k,w} = i_* \psi_{k,w}, \Phi_k(w) = i_* \Psi_k(w)) \\
            (D, F_{k,w}, \xi_{k,w}, \Xi_k(w)) \arrow[hookrightarrow]{r}{i_2} & (\Sigma, (i_2)_* F_{k,w}, (i_2)_* \xi_{k,w}, (i_2)_* \Xi_k(w)) \\
        \end{tikzcd}\end{center}
        
        All these data will be explicitly defined in Subsection~\ref{sec:3.2.1}. \\
        
        With these constructions in hand, we will prove new versions of Propositions~\ref{prop:main},~\ref{prop:main2}, and these will imply the generalized Theorems~\ref{thm:2},~\ref{thm:3}. These dependencies are depicted in Figure~\ref{fig:new_depends}. The deductions marked $**$ can be taken as is from the proofs for surfaces of genus $\geq 4$ in~\cite{PS14},~\cite{10}. The deductions marked $*$ will now require justification, since $i_2$ will no longer be incompressible. Since Claims~\ref{cl:2},~\ref{cl:2.2} were already established (see proof in~\cite{PS14},~\cite{10} or sketch in Subsection~\ref{sec:3.1}), we have to show why they imply Propositions~\ref{prop:new},~\ref{prop:new2}. This is done in Subsection~\ref{sec:3.2.2}, and uses Lemma~\ref{lemma:2}.
        
        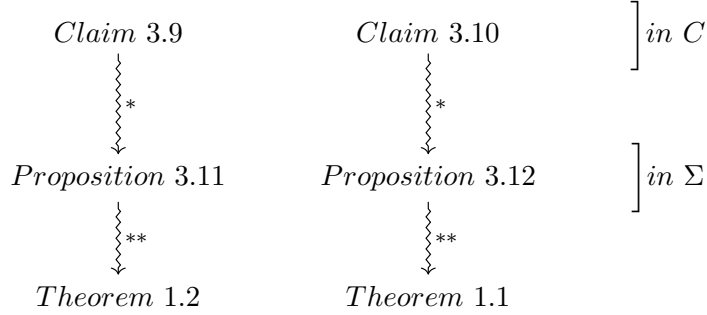
\begin{figure}
            \centering
            \begin{tikzcd}
                Claim~\ref{cl:2} \arrow[d,rightsquigarrow,"*"] & Claim~\ref{cl:2.2} \arrow[d,rightsquigarrow,"*"] & \bigg] in \ C \\
                Proposition~\ref{prop:new} \arrow[d,rightsquigarrow,"**"] & Proposition~\ref{prop:new2} \arrow[d,rightsquigarrow,"**"] & \bigg] in \ \Sigma \\
                Theorem~\ref{thm:3} & Theorem~\ref{thm:2}
            \end{tikzcd}
            \caption{Dependencies between generalized intermediate results.}
            \label{fig:new_depends}
        \end{figure}

        Propositions~\ref{prop:new},~\ref{prop:new2}, which by the above discussion imply the main theorems of this paper, are stated here.
        
        \begin{prop}
        \label{prop:new}
            Let $w = \Pi_{j=1}^r V^{N_j} H^{M_j} \in F_2 = \langle H,V \rangle$ be balanced. For large enough $k \in \N$, there are $2^{2r}$ non-degenerate fixed points of $(i_2)_* \Xi_k(w)$ in $\Sigma$ whose orbits have free homotopy class $\tau_{i_2}(\alpha_{k,w})$, and such fixed points have actions and Conley-Zehnder indices as in Equations~\ref{eq:2},\ref{eq:3}.
        \end{prop}
        
        \begin{prop}
        \label{prop:new2}
            Let $w = (VH)^r \in F_2$ for some $r \in \N$. For large enough $k \in K$, there are $2^{2r}$ fixed points of $(i_2)_* \Xi_k(w)$ in $\Sigma$ whose orbits have free homotopy class $\tau_{i_2}(\alpha_k^\prime)$, and such fixed points in different orbits have action gaps that grow linearly with $k$: for such fixed points in different orbits $y, z$:
            \[
                | \mathcal{A}(y) - \mathcal{A}(z) | \geq c \cdot k + O(1) ,
            \]
            as $k \to \infty$, for some global constant $c > 0$. The set $K \subset \N$ is as defined in Subsection~\ref{sec:3.1}.
        \end{prop}
        
        Note that since $\pi_0(\mathscr{L}C) \subset \pi_0(\mathscr{L}D)$, the expressions $\tau_{i_2}(\alpha_{k,w})$ and $\tau_{i_2}(\alpha_k^\prime)$ are well-defined.
        
        \subsubsection{The construction}
        \label{sec:3.2.1}
            In a sense, the proof for $\Sigma$ of genus 2 is harder than for genus 3. In the following, we will show the proof for $\Sigma$ of genus 2, and comment on the differences to the proof for genus 3 surfaces when they arise.
            
            Consider $C$, defined in Subsection~\ref{sec:geometric}, and two additional copies of the cylinder $C_*$, $C_1$ and $C_2$. Our manifold $D$ will be $D = C \bigsqcup C_1 \bigsqcup C_2$, equipped with the standard symplectic form $\omega_0 = dx \wedge dy$ on every component. Denote the symplectic inclusions $C_* \hookrightarrow C_1,C_2$ by $c_1,c_2$ respectively. Note that the additional annuli, $C_1$ and $C_2$, are needed to make sure that $i_2$ and $F_{k,w}$ satisfy Condition~\ref{eq:condition}; that is, to enable $F_{k,w} \circ i_2^{-1}: \Ima(i_2) \to \R$ to be extended by a locally constant function to a function on all of $\Sigma$.
            
            The symplectic embedding $i_2:(D,\omega_0) \hookrightarrow (\Sigma, \sigma)$ is built in stages. First define $i_2$ on $C$: embed $C$ symplectically into $\R^2$, embed the plane $\R^2$ into $S^2$, and then add 2 handles as shown in Figure~\ref{fig:handles}; that is, $C$ seperates the sphere into 4 connected components, each having two neighbors; connect any two non-neighboring components with a handle. In the case of genus 3, add a handle inside the 'outside' component to obtain a surface of genus 3 (we will refer to this handle as the extra handle). This defines $i_2 \restriction_C$. Define $i_2$ on $C_1 \bigsqcup C_2$ by embedding each of them symplectically into one of the non-extra handles (with $C_1$ and $C_2$ on different handles), such that the images $i_2(C_1), i_2(C_2)$ are not contractible. The orientation of the embeddings $i_2\restriction_{C_1}, i_2\restriction_{C_2}$, and which goes on which handle, will be defined immediately. The embedding $i_2$ can be seen in Figure~\ref{fig:handles}.
            
            \begin{figure}[!ht]
                \centering
                \begin{minipage}{.5\textwidth}
                    \centering
                    \scalebox{0.75}{
                        \begin{tikzpicture}
                            \draw[gray,ultra thin] (-1,0) circle (3);
                            \draw[gray,ultra thin] (-1,0) circle (2.5);
                            \draw (-1.5,3) node[anchor=south] {\Large $i_2(C_V)$};
                            
                            \draw[gray,ultra thin] (1,0) circle (3);
                            \draw[gray,ultra thin] (1,0) circle (2.5);
                            \draw (1.5,3) node[anchor=south] {\Large $i_2(C_H)$};
                            
                            \filldraw[fill=white,draw=black,thick]
                                (0,-2) ++(35:3.7) arc (35:145:3.7) -- 
                                ++(-35:0.6) arc (145:35:3.1) -- cycle;
                            \filldraw[fill=white,draw=black,thick] (-2.75,0) circle (0.3);
                            \filldraw[fill=white,draw=black,thick] (2.75,0) circle (0.3);
                            
                            \filldraw[fill=white,draw=black,thick]
                                (-1.5,-3) ++(54:2.8) arc (54:-54:2.8) --
                                ++(126:0.6) arc (-54:54:2.2) -- cycle;
                            \filldraw[fill=white,draw=black,thick] (0,-1) circle (0.3);
                            \filldraw[fill=white,draw=black,thick] (0,-5) circle (0.3);
                            
                            \draw[gray, ultra thin] (-1.5,-3) ++(-20:2.5) ++(-20:0.3) arc (-20:-200:0.3);
                            \draw[gray, dashed] (-1.5,-3) ++(-20:2.5) ++(160:0.3) arc (160:-20:0.3);
                            \draw[gray, ultra thin] (-1.5,-3) ++(-22:2.5) ++(-22:0.3) arc (-22:-202:0.3);
                            \draw[gray, dashed] (-1.5,-3) ++(-22:2.5) ++(158:0.3) arc (158:-22:0.3);
                            
                            \draw[gray, ultra thin] (0,-2) ++(90:3.4) ++(90:0.3) arc (90:270:0.3);
                            \draw[gray, dashed] (0,-2) ++(90:3.4) ++(270:0.3) arc (270:450:0.3);
                            \draw[gray, ultra thin] (0,-2) ++(92:3.4) ++(92:0.3) arc (92:272:0.3);
                            \draw[gray, dashed] (0,-2) ++(92:3.4) ++(272:0.3) arc (272:452:0.3);
                            
                            \draw (0,1) node[anchor=north] {\Large $i_2(C_1)$};
                            \draw (1.2,-4) node[anchor=west] {\Large $i_2(C_2)$};
                            
                            \filldraw[black] (0,-2.55) circle (2pt);
                            \draw (0,-2.8) node[anchor=north] {\Large $i_2(s_0)$};
                        \end{tikzpicture}
                    }
                    \caption{adding handles on $S^2$ after embedding $C$: this whole figure is in $S^2$.}
                    \label{fig:handles}
                \end{minipage}%
                \begin{minipage}{.5\textwidth}
                    \centering
                    \pgfdeclarelayer{bg}
                    \pgfsetlayers{bg,main}
                    \hspace{-2.5cm}
                    \scalebox{0.9}{
                        \begin{tikzpicture}
                            \filldraw[fill=white,draw=black,thick]
                                (0,-2) ++(35:3.7) arc (35:145:3.7) -- 
                                ++(-35:0.6) arc (145:35:3.1) -- cycle;
                            \filldraw[fill=white,draw=black,thick] (-2.75,0) circle (0.3);
                            \filldraw[fill=white,draw=black,thick] (2.75,0) circle (0.3);
                            
                            \filldraw[fill=white,draw=black,thick]
                                (-1.5,-3) ++(54:2.8) arc (54:-54:2.8) --
                                ++(126:0.6) arc (-54:54:2.2) -- cycle;
                            \filldraw[fill=white,draw=black,thick] (0,-1) circle (0.3);
                            \filldraw[fill=white,draw=black,thick] (0,-5) circle (0.3);
                            
                            \node (origin) at (0,-2.55) {};
                            \node (A) at ($(0,-2) + (145:3.4)$) {};
                            \node (B) at ($(0,-2) + (35:3.4)$) {};
                            \node (C) at ($(-1.5,-3) + (54:2.6)$) {};
                            \node (D) at ($(-1.5,-3) + (-54:2.6)$) {};
                            
                            \filldraw[black] (origin) circle (2pt);
                            \draw (0,-2.8) node[anchor=north] {\Large $i_2(s_0)$};
                            
                            \begin{scope}
                            [gray,dashed,decoration={markings,mark=at position 0.55 with {\arrow[scale=2]{>}}}]
                                \draw
                                    (origin) to[out=140,in=235] (A) arc (145:35:3.5) to[out=-55,in=0] (0,0.5) to[out=180,in=122,distance=50] (origin)
                                    decorate {(A) to[out=55,in=130] (B)};
                                \draw
                                    (origin) to[out=100,in=140,distance=30] (C) arc (54:-54:2.6) to[out=215,in=215,distance=60] (origin)
                                    decorate {(D) arc (-54:54:2.2)};
                            \end{scope}
                            
                            \draw (-4.3,0) node {\Large $g_1$};
                            \draw (0,2) node {\Large $g_2$};
                            \draw (1.5,-1) node {\Large $g_3$};
                            \draw (1.5,-3.5) node {\Large $g_4$};
                    
                            \begin{pgfonlayer}{bg}
                                \begin{scope}
                                [gray,decoration={markings,mark=at position 0.4 with {\arrow[scale=2]{>}}}]
                                    \draw[postaction={decorate}] (origin) to[out=160,in=123,distance=200] (origin);
                                    \draw[postaction={decorate}] (origin) to[out=120,in=45,distance=130] (origin);
                                \end{scope}
                            \end{pgfonlayer}
                        \end{tikzpicture}
                        }
                    \caption{the generators of $\pi_1(\Sigma, i_2(s_0))$, dashed lines only for readability.}
                    \label{fig:gens}
                \end{minipage}
            \end{figure}
            
            Recall $s_0 \in C$, and $a,b,c$, the generators of $\pi_1(C, s_0)$. Fix some $s^\prime_k \in C_k$ for $k=1,2$. The generators of $\pi_1(C_k,s^\prime_k) = \Z$ are denoted $d_k$. These are positively oriented - that is, $d_k = [t\mapsto(0,Lt) \in C_k]$ as elements in $\pi_1(C_k,s_k^\prime)$. Denote the generators of $\pi_1(\Sigma, i_2(s_0))$ by $g_1,..,g_4$, such that $\pi_1 (\Sigma, i_2(s_0)) = \left\langle g_1,g_2,g_3,g_4\middle|[g_1,g_2] [g_3,g_4] \right\rangle$, and the loops $g_1$ and $g_3$ go around the two handles (see Figure~\ref{fig:gens}). In the case where $\Sigma$ is of genus 3, denote $\pi_1(\Sigma, i_2(s_0)) = \langle g_1, ..., g_6 | [g_1,g_2] [g_3,g_4] [g_5,g_6] \rangle$, such that the loops $g_1$ and $g_3$ still go around the two non-extra handles (the same as in Figure~\ref{fig:gens}), and such that $g_5$ goes around the extra handle. The loops $g_{k+1}$ go along the handle that the loops $g_k$ go around (for $k = 1,3,5$), oriented such that $[g_1,g_2][g_3,g_4][g_5,g_6] = 1$ in $\pi_1(\Sigma, i_2(s_0))$.
            
            Choose the image of $C_1$ to be contained in the handle that the loop $g_1 \in \pi_1(\Sigma, i_2(s_0))$ goes around (in Figures~\ref{fig:handles} and~\ref{fig:gens} this is the upper handle), and so $i_2(C_2)$ is contained in the handle that the loop $g_3$ goes around. Since the image of $C_1$ is located on a handle, is non-contractible, and $i_2$ is an embedding, one can choose how to orient the image of $C_1$:
            \[
                \tau_{i_2} \circ \eta_{C_1,s_1^\prime} (d_1) = \eta_{\Sigma, i_2(s_0)} (g_1) \text{ or } \tau_{i_2} \circ \eta_{C_1,s_1^\prime} (d_1) = \eta_{\Sigma, i_2(s_0)} (g_1^{-1}) ,
            \]
            and similiarly one can choose the orientation of $i_2(C_2)$:
            \[
                \tau_{i_2} \circ \eta_{C_2,s_2^\prime} (d_2) = \eta_{\Sigma, i_2(s_0)} (g_3) \text{ or } \tau_{i_2} \circ \eta_{C_2,s_2^\prime} (d_2) = \eta_{\Sigma, i_2(s_0)} (g_3^{-1}) .
            \]
            Choose the orientation of the embeddings $i_2\restriction_{C_1}, i_2\restriction_{C_2}$ to be such that the elements $d_1,d_2$ map to $\eta_{\Sigma, i_2(s_0)} (g_1)$ and $\eta_{\Sigma, i_2(s_0)} (g_3^{-1})$, respectively, under $\tau_{i_2} \circ \eta_{C_k, s_k^\prime}$ (for $k = 1,2$). This concludes the definition of the embedding $i_2$. \\
            
            By construction of $i_2: D \hookrightarrow \Sigma$, the pushforward $(i_2)_*: \pi_1(D, s_0) \to \pi_1(\Sigma, i_2(s_0))$ acts on the generators $a,b,c$ of $\pi_1(D,s_0)$ so:
            \begin{myequation}
                a \mapsto g_1 g_3 , \\
                b \mapsto g_2 g_1^{-1} g_2^{-1} g_3 , \\
                c \mapsto g_3 .
            \end{myequation}
            
            In order to define the Hamiltonian $F_{k,w}$ on $D$, a similar procedure as the one that defines $G_{k,w}$ is followed. Recall $h = h_k: C_* \to \R$, the Hamiltonian on $C_*$:
            \[
                (x,[y]) \mapsto -\frac{k}{2} + k \int_{-1}^x u(s) ds ,
            \]
            defined in Subsection~\ref{sec:3.1}. Note that the pairs $e_2 \circ c_V, h_k$ and $e_2 \circ c_H, h_k$ satisfy Condition~\ref{eq:condition}. Define two autonomous Hamiltonians on $D$:
            \begin{myequation}
                g_{k,V}: D \to \R , \\
                g_{k,V} = (e_2 \circ c_V)_* h_k \bigsqcup - (c_1)_* h_k \bigsqcup (c_2)_* h_k ; \\
                g_{k,H}: D \to \R , \\
                g_{k,H} = (e_2 \circ c_H)_* h_k \bigsqcup (c_1)_* h_k \bigsqcup (c_2)_* h_k . \\
            \end{myequation}
            
            Given $w \in F_2 = \langle V,H \rangle$, define $\xi_{k,w}(t)$, the flow on $D$, as the concatenation of the flows induced by $g_{k,V}, g_{k,H}$, in the order induced by $w$. The Hamiltonian which induces this flow is denoted $F_{k,w}: S^1 \times D \to \R$, and its time-one-map is $\Xi_k(w): D \to D$. Another equivalent way to define $\Xi_k(w)$ is to denote the time-one-maps of $g_{k,V}$, $g_{k,H}$ by $g_{k,V}^\prime, g_{k,H}^\prime$ respectively, and define the homomorphism
            \begin{myequation}
                \Xi_k: F_2 \to Ham_c(D,\omega_0) , \\
                V \mapsto g_{k,V}^\prime, H \mapsto g_{k,H}^\prime .
            \end{myequation}
            
            Next, note that $F_{k,w}$ and $i_2$ satisfy Condition~\ref{eq:condition} for all $k,w$. Therefore we can define the pushforwards $(i_2)_* F_{k,w}: S^1 \times \Sigma \to \R, (i_2)_* \xi_{k,w}: \R \times \Sigma \to \Sigma, (i_2)_* \Xi_k(w): \Sigma \to \Sigma$.
        
        \subsubsection{Proof of Propositions~\ref{prop:new},~\ref{prop:new2}}
        \label{sec:3.2.2}
            First we state a few helper claims, and use them to prove Propositions~\ref{prop:new},~\ref{prop:new2}. Then, we will show the proofs of the claims. For the rest of this subsection, fix some $2 \leq k \in \N, w = V^{N_1} H^{M_1} ... V^{N_r} H^{M_r} \in F_2$ balanced. Recall that $\alpha_{k,w} = \eta_{C,s_0}(\beta_{k,w})$, $\alpha_k^\prime = \eta_{C,s_0}(\beta_k^\prime)$, and also note that since the connected component of $s_0$ in $D$ is $e_2(C)$, we shall abuse notation and denote $\pi_1(C,s_0) = \pi_1(D,s_0), \eta_{C,s_0} = \eta_{D,s_0}$, etc.
            
            The first claim characterizes all egg-beater orbits in $C$ in terms of their free homotopy classes.
            
            \begin{claim}
            \label{claim:3}
                Let $\tilde{\gamma} : S^1 \to C$ be a closed $\psi_{k,w}$-orbit. Then there exist $k_1,...,k_r, l_1,...,l_r \in \Z$, $u_1,...,u_r, v_1,...,v_r \in \{0,1\}$ such that $[\tilde{\gamma}]_{\pi_0(\mathscr{L}C)} = \eta_{C,s_0}(\Pi_{m=1}^r a^{k_m}c^{-u_m}b^{l_m}c^{v_m})$.
            \end{claim}
            
            The second claim provides a partial injectivity property of $\tau_{i_2}$.
            
            \begin{claim}
            \label{claim:4}
                Let $\gamma = \eta_{D,s_0}(\Pi_{m=1}^r a^{k_m}c^{-u_m}b^{l_m}c^{v_m}) \in \pi_0(\mathscr{L}D)$ for some $k_1,...,k_r,l_1,...,l_r \in \Z$, $u_1,...,u_r,v_1,...,v_r \in \{0,1\}$, and assume $\tau_{i_2}(\gamma) = \tau_{i_2}(\alpha_{k,w})$. Then $\gamma = \alpha_{k,w}$.
                
                Alternatively, if $k \in K$, assume $\tau_{i_2}(\gamma) = \tau_{i_2}(\alpha_k^\prime)$. Then $\gamma = \alpha_k^\prime$.
            \end{claim}
            
            The third claim will allow us to focus our calculations to $\xi_{k,w}$-orbits in $D$ which are in fact in $e_2(C)$.
            
            \begin{claim}
            \label{claim:5}
                Let $k \in \N$ be large enough, and let $z \in \Sigma$ be a non-degenerate fixed point of $(i_2)_* \Xi_k(w)$ whose $(i_2)_* \xi_{k,w}$-orbit is in the class $\tau_{i_2}(\alpha_{k,w})$, i.e. $[(i_2)_* \xi_{k,w}(t) (z)]_{\pi_0(\mathscr{L}\Sigma)} = \tau_{i_2}(\alpha_{k,w})$.
                
                Alternatively, let $k \in K$ be large enough, and let $z \in \Sigma$ be a non-degenerate fixed point of $(i_2)_* \Xi_k(w)$ whose $(i_2)_* \xi_{k,w}$-orbit is in the class $\tau_{i_2}(\alpha_k^\prime)$.
                
                In both cases, $z \in i_2(C)$. \\
            \end{claim}
            
            \begin{proof}[Proof of Propositions~\ref{prop:new},~\ref{prop:new2}]
                Recall that Claims~\ref{cl:2},~\ref{cl:2.2}, proven in~\cite{10},~\cite{PS14}, are analogous to these propositions, but set in $C$. We want to show why they imply that these propositions hold in $\Sigma$. The proofs for both propositions are very similar, so we present them simultaneously.
                
                Let $\tilde{\gamma}: S^1 \to \Sigma$ be a closed orbit of $(i_2)_* \xi_{k,w}(t)$ in the class $\tau_{i_2}(\alpha_{k,w})$ (or $\tau_{i_2}(\alpha_k^\prime)$, in the case where $k \in K$) such that $z = \tilde{\gamma}(0)$ is a non-degenerate fixed point of $(i_2)_* \Xi_k(w)$. By Claim~\ref{claim:5}, if $k$ is large enough then $z \in i_2(C)$, therefore $y = i_2^{-1}(z) \in C$ is uniquely defined. Denote by $\tilde{\gamma_y} : S^1 \to C$ the $\xi_{k,w}(t)$-orbit of $y$. By Claim~\ref{claim:3}, there exist $k_1,...,k_r, l_1,...,l_r \in \Z , u_1,...,u_r, v_1,...,v_r \in \{0,1\}$ such that
                \[
                    [\tilde{\gamma_y}]_{\pi_0(\mathscr{L}C)} = \eta_{C,s_0}(\Pi_{m=1}^r a^{k_m}c^{-u_m}b^{l_m}c^{v_m}) .
                \]
                
                Since $\tau_{i_2}([\tilde{\gamma_y}]_{\pi_0(\mathscr{L}C)}) = [\tilde{\gamma}]_{\pi_0(\mathscr{L}\Sigma)} = \tau_{i_2}(\alpha_{k,w})$ (or $\tau_{i_2}(\alpha_k^\prime)$), one has $[\tilde{\gamma_y}]_{\pi_0(\mathscr{L}C)} = \alpha_{k,w}$ (or $[\tilde{\gamma_y}]_{\pi_0(\mathscr{L}C)} = \alpha_k^\prime$) by Claim~\ref{claim:4}.
                
                Thus non-degenerate fixed points of $(i_2)_* \Xi_k(w)$ in $\Sigma$ whose orbits are in the class $\tau_{i_2}(\alpha_{k,w})$ correspond in a 1-1 manner to non-degenerate fixed points of $\Xi_k(w)$ in $D$ whose orbits are in the class $\alpha_{k,w}$ (and there is a similar correspondence for the classes $\tau_{i_2}(\alpha_k^\prime),\alpha_k^\prime$). Note that since Claim~\ref{claim:4} uses Lemma~\ref{lemma:2}, this is the point where the algebraic Lemma~\ref{lemma:2} enters the proof. \\
                
                Therefore it is enough to restrict the analysis to the dynamics on $D$, and further, to the dynamics on $C$, since $z \in i_2(C)$. By Claim~\ref{cl:2}, for large enough $k$ there are $2^{2r}$ non-degenerate fixed points of $\Xi_k(w)$ in the class $\alpha_{k,w}$, indexed by $\vec{\epsilon} = (\epsilon_0,...,\epsilon_{2r-1}) \in \{\pm 1\}^{2r}$, such that their actions and Conley-Zehnder indices are given by Equations~\ref{eq:2},\ref{eq:3}. Since $i_2$ preserves actions and Conley-Zehnder indices by construction, this proves Proposition~\ref{prop:new}.
                
                Similarly, setting $w = (VH)^r$, by Claim~\ref{cl:2.2} for large enough $k \in K$ there are $2^{2r}$ non-degenerate fixed points of $\Xi_k(w)$ in the class $\alpha_k^\prime$, such that fixed points in different orbits have action gaps that grow linearly with $k$. The injection $i_2$ preserves actions, so this proves Proposition~\ref{prop:new2}.
            \end{proof}
            
            We finish this section with the proofs of the claims.
            
            \begin{proof}[Proof of Claim~\ref{claim:3}]
                Denote $z = \tilde{\gamma}(0)$. Recall that $\psi_{k,w}^\prime$ is the piecewise-linear version of $\psi_{k,w}$ defined in Subsection~\ref{sec:3.1}, and that it is the concatenation of some autonomous Hamiltonian isotopies $f_{0,V}^{N_1t}, f_{0,H}^{M_1t}, ...$; while $\psi_{k,w}$ is their composition. Thus $\tilde{\gamma}$ and $t \mapsto \psi_{k,w}^\prime(t) (z)$ are freely homotopic loops in $C$.
                
                Therefore, it is enough to show there exist some $k_1,...,k_r, l_1,...,l_r \in \Z$, $u_1,...,u_r, v_1,...,v_r \in \{0,1\}$ such that $[t \mapsto \psi_{k,w}^\prime(t) (z)]_{\pi_0(\mathscr{L}C)} = \eta_{C,s_0}(\Pi_{m=1}^r a^{k_m} c^{-u_m} b^{l_m} c^{v_m})$, since homotopic loops have the same $\eta_{C,s_0}$-images. This is shown in the proof of Lemma 4.2 in~\cite{10} - in their notation, $\psi_{k,w}^\prime(t)$ is denoted $\phi^t$, and $k_\mu, l_\mu, u_\mu, v_\mu$ are denoted $n_\mu, m_\mu, \epsilon_\mu, \nu_\mu$ respectively, for $1 \leq \mu \leq r$ (note we do not need the whole statement of Lemma 4.2).
            \end{proof}
            
            \begin{proof}[Proof of Claim~\ref{claim:4}]
                The two assumptions are similar, and the proofs for the two cases are the same. For concreteness, assume $\tau_{i_2}(\gamma) = \tau_{i_2}(\alpha_{k,w})$, the other case is analogous.
                
                Let $\delta \in \pi_1(C,s_0)$ be any element such that $\gamma = \eta_{C,s_0}(\delta)$. By assumption, $\tau_{i_2}(\gamma) = \tau_{i_2}(\alpha_{k,w})$. Commutativity implies $\eta_{\Sigma,i_2(s_0)}((i_2)_* \delta) = \eta_{\Sigma,i_2(s_0)}((i_2)_* \beta_{k,w})$, so one sees that $(i_2)_* \delta$ and $(i_2)_* \beta_{k,w}$ are conjugate in $\pi_1(\Sigma, i_2(s_0))$. 
                
                In the case where the genus of $\Sigma$ is 2, recall that $\pi_1(D,s_0) = F_3$, and that $(i_2)_* : F_3 \to \pi_1(\Sigma, i_2(s_0))$ is exactly the same homomorphism as in the requirements of Lemma~\ref{lemma:2}.
                
                In the case where the genus of $\Sigma$ is 3, denote
                \begin{myequation}
                    p_{1234}: \pi_1(\Sigma, i_2(s_0)) \to \langle g_1,g_2,g_3,g_4 | [g_1,g_2][g_3,g_4] \rangle , \\
                    g_1 \mapsto g_1, g_2 \mapsto g_2, g_3 \mapsto g_3, g_4 \mapsto g_4 , \\
                    g_5,g_6 \mapsto 1 , 
                \end{myequation}
                and note that $p_{1234} \circ (i_2)_*$ is exactly the homomorphism as in the requirements of Lemma~\ref{lemma:2}. Since $(i_2)_* \delta, (i_2)_* \beta_{k,w}$ are conjugate in $\pi_1(\Sigma, i_2(s_0))$, so are $(p_{1234} \circ (i_2)_*) (\delta), (p_{1234} \circ (i_2)_*) (\beta_{k,w})$ in $\langle g_1,g_2,g_3,g_4 | [g_1,g_2][g_3,g_4] \rangle$.
                
                In both cases, one can get by Lemma~\ref{lemma:2} that $\delta$ and $\beta_{k,w}$ are conjugate in $\pi_1(D,s_0)$, so $\gamma = \eta_{D,s_0}(\delta) = \eta_{D,s_0}(\beta_{k,w}) = \alpha_{k,w}$ by definition of $\eta_{D,s_0}$.
            \end{proof}
            
            \begin{proof}[Proof of Claim~\ref{claim:5}]
                Under both of the assumptions (that the orbit of $z$ is in the class $\tau_{i_2}(\alpha_{k,w})$ or $\tau_{i_2}(\alpha_k^\prime)$), the orbit of $z$ is in the class $\eta_{\Sigma, i_2(s_0)}(\Pi_{m=1}^r (g_1 g_3)^{k_{2m-1}} (g_2 g_1^{-1} g_2^{-1} g_3)^{k_{2m}})$ for some sequence $k_1,...,k_{2r} \in \Z$, since $(i_2)_*(a) = g_1 g_3, (i_2)_*(b) = g_2 g_1^{-1} g_2^{-1} g_3$. Note that for large enough $k$, the elements $k_m$ in this sequence all satisfy $|k_m| \geq 2$.
                
                As noted in Section~\ref{sec:2}, $\pi_1(\Sigma)$ is a free product of $H_1 = \langle g_1,g_2 \rangle, H_2 = \langle g_3,g_4 \rangle$ with amalgamation of subgroups $A = \langle [g_1,g_2] \rangle, B = \langle [g_3,g_4] \rangle$ by the isomorphism $\phi: A \to B: [g_1,g_2] \mapsto [g_3,g_4]^{-1}$. In the case where the genus of $\Sigma$ is 3, $\pi_1(\Sigma)$ is a free product of $H_1 = \langle g_1,g_2 \rangle, H_2 = \langle g_3,g_4,g_5,g_6 \rangle$ with amalgamation of subgroups $A = \langle [g_1,g_2] \rangle, B = \langle [g_3,g_4][g_5,g_6] \rangle$ by the isomorphism $\phi : A \to B : [g_1,g_2] \mapsto ([g_3,g_4][g_5,g_6])^{-1}$.
                
                Note that $z \in i_2(D)$, else $z$ is a degenerate fixed point. Assume by contradiction $z \in i_2(C_1 \cup C_2)$. Denote the orbit of $z$ under $(i_2)_* \xi_{k,w}(t)$ by $\tilde{\gamma}: S^1 \to \Sigma$. 
                
                If $z \in i_2(C_1)$, then $\tilde{\gamma}(t) \in i_2(C_1)$ for all $t$, since $\tilde{\gamma}(t) \in i_2(D)$ for all $t$, and $C_1$ is a connected component of $D$. This means $\tilde{\gamma}$ is (freely) homotopic to some loop $\tilde{\delta}: S^1 \to \Sigma$ based in $i_2(s_0)$ with $[\tilde{\delta}]_{\pi_1(\Sigma,i_2(s_0))} = g_1^n$ for some $n \in \Z$, since by construction $\tau_{i_2} \circ \eta_{C_1, s_1^\prime} (d_1) = \eta_{\Sigma,i_2(s_0)} (g_1)$ (recall that $d_1$ is the homotopy class of a loop going around $i_2(C_1)$ once). Denote the homotopy between $\tilde{\gamma},\tilde{\delta}$ by $F : S^1 \times [0,1] \to \Sigma$ with $F(\cdot,0) = \tilde{\gamma}, F(\cdot, 1) = \tilde{\delta}$. Then $\rho = F(0, \cdot)$ is a loop in $\Sigma$ based in $i_2(s_0)$, and $[\tilde{\gamma}]_{\pi_1(\Sigma,i_2(s_0))} = [\rho \# \tilde{\delta} \# \overline{\rho}]_{\pi_1(\Sigma, i_2(s_0))}$. The following holds:
                \begin{myequation}
                    \Pi_{m=1}^r (g_1 g_3)^{k_{2m-1}} (g_2 g_1^{-1} g_2^{-1}g_3)^{k_{2m}} \sim \\
                    \sim [\tilde{\gamma}]_{\pi_1(\Sigma, i_2(s_0))} = [\rho \# \tilde{\delta} \# \overline{\rho}]_{\pi_1(\Sigma, i_2(s_0))} = \chi g_1^n \chi^{-1} ,
                \end{myequation}
                where $\chi = [\rho]_{\pi_1(\Sigma, i_2(s_0))}$. Thus $\Pi_{m=1}^r (g_1 g_3)^{k_{2m-1}} (g_2 g_1^{-1} g_2^{-1}g_3)^{k_{2m}} \sim g_1^n$.
                
                Note that $(g_1^n)$, a sequence with a single element, is a cyclically reduced sequence of $\langle H_1*H_2 , A=B, \phi \rangle$, so $g_1^n$ is a cyclically reduced element, with $len(g_1^n) = 1$. Recall that $w$ is given in cyclically reduced form, thus $\forall \ 1 \leq m \leq r: M_m, N_m \neq 0$. Since for all $m$, $|k_m| \geq 2$, by Claim~\ref{cl:length} it can be seen that $len \left(\Pi_{m=1}^r (g_1 g_3)^{k_{2m-1}} (g_2 g_1^{-1} g_2^{-1}g_3)^{k_{2m}} \right) > 1$, in contradiction. If the genus of $\Sigma$ is 3, then note that $p_{1234}$ (defined in the proof of Claim~\ref{claim:4}) preserves factors (in the sense of the proof of Claim~\ref{cl:length}), which means that $len(\Pi_{m=1}^r (g_1 g_3)^{k_{2m-1}} (g_2 g_1^{-1} g_2^{-1}g_3)^{k_{2m}}) = len(p_{1234}(\Pi_{m=1}^r (g_1 g_3)^{k_{2m-1}} (g_2 g_1^{-1} g_2^{-1}g_3)^{k_{2m}})) > 1$, by Claim~\ref{cl:length}. This is again a contradiction. Therefore $z \not\in i_2(C_1)$.

                Likewise, assume $z \in i_2(C_2)$. Similarly to the previous case, there is a conjugacy between $\Pi_{m=1}^r (g_1 g_3)^{k_{2m-1}} (g_2 g_1^{-1} g_2^{-1}g_3)^{k_{2m}}$ and $g_3^n$ in $\pi_1(\Sigma, i_2(s_0))$ for some $n \in \Z$. As before, $len(\Pi_{m=1}^r (g_1 g_3)^{k_{2m-1}} (g_2 g_1^{-1} g_2^{-1}g_3)^{k_{2m}}) > 1$ and $len(g_3^n) = 1$, but $len$ is conjugation-invariant, so this is a contradiction.
                
                All cases lead to a contradiction, therefore $z \in i_2(C)$.
            \end{proof}

\section{Incompressibility in genus 3}
\label{sec:incompressibility}
    This section describes an embedding $i_3$ of an egg-beater-like surface $E = C \bigsqcup C_1 \bigsqcup C_2 \bigsqcup C_3$ to a closed orientable surface $\Sigma_3$ of genus 3 such that $i_3 \restriction_C$ is incompressible, i.e. $(i_3 \restriction_C)_*: \pi_1(C) \to \pi_1(\Sigma_3)$ and $\tau_{i_3 \restriction_C}: \pi_0(\mathscr{L}C) \to \pi_0(\mathscr{L}\Sigma_3)$ are both injective. This incompressibility will imply Theorems~\ref{thm:2},~\ref{thm:3} by following the original proofs in~\cite{PS14} and~\cite{10}. Notation from previous sections is used in this section.
    
    \subsection{The egg-beater surface and map, and the embedding}
        Recall the egg-beater surface $D = C_V \bigcup_{VH_{0,1}} C_H \bigsqcup C_1 \bigsqcup C_2$ defined in Section~\ref{sec:3}, and the cylinder $C_* = [-1,1] \times \R / L\Z$. Let $C_3$ be another copy of $C_*$, and define $E = D \bigsqcup C_3$, with the symplectic form $dx \wedge dy$ on every component. Denote by $c_3: C_* \to E, e_3: D \to E$ the symplectic inclusions.
        
        The embedding $i_3: E \hookrightarrow \Sigma_3$ is defined in two parts. First, embed $C$ symplectically into $S^2$. Note that $C$ separates the sphere into 4 connected components. Choose any one of these components, and connect it with handles to the other 3 components, with one handle each. The result is a genus 3 surface $\Sigma_3$, and this defines $i_3 \restriction_C$. Define $i_3$ on $C_1 \bigsqcup C_2 \bigsqcup C_3$ by symplectically embedding each of them on a different handle. The embedding $i_3$ can be seen in Figure~\ref{fig:handles_gen_3}, and the generators of $\pi_1(\Sigma_3,i_3(s_0)) = \langle g_1,...,g_6 | [g_1,g_2][g_3,g_4][g_5,g_6] \rangle$ can be seen in Figure~\ref{fig:gens_gen_3}. The orientations of $i_3 \restriction_{C_j}$ for $j=1,2,3$ are chosen such that
        \begin{myequation}
            \tau_{e_3 \circ c_1} [t \mapsto (0,Lt)]_{\pi_0(\mathscr{L}C_*)} = \eta_{\Sigma_3,i_3(s_0)}(g_1) ,\\
            \tau_{e_3 \circ c_2} [t \mapsto (0,Lt)]_{\pi_0(\mathscr{L}C_*)} = \eta_{\Sigma_3,i_3(s_0)}(g_3) ,\\
            \tau_{c_3} [t \mapsto (0,Lt)]_{\pi_0(\mathscr{L}C_*)} = \eta_{\Sigma_3,i_3(s_0)}(g_5) .
        \end{myequation}
        
        The various cylinders, egg-beater surfaces and closed surfaces mentioned in this paper and the embeddings between them are summarized in Figure~\ref{fig:embeddings}.
        
        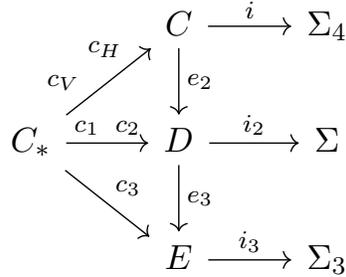
\begin{figure}[!ht]
            \centering
            \adjustbox{scale=1.2,center}{
                \begin{tikzcd}
                    & C \arrow[r,"i"] \arrow[d,"e_2"] & \Sigma_4 \\
                    C_* \arrow[ru,"c_V" near start,"c_H" near end] \arrow[r,"c_1" near start,"c_2" near end]\arrow[rd,"c_3"] & D \arrow[r,"i_2"] \arrow[d,"e_3"] & \Sigma \\
                    & E \arrow[r,"i_3"] & \Sigma_3
                \end{tikzcd}
            }
            \caption{Embeddings between the surfaces.}
            \label{fig:embeddings}
        \end{figure}
        
        \begin{figure}[!ht]
            \centering
            \hspace{-0.05\textwidth}%
            \begin{minipage}{.5\textwidth}
                \centering
                \scalebox{0.8}{
                \begin{tikzpicture}
                    \draw[gray,ultra thin] (-1,0) circle (3);
                    \draw[gray,ultra thin] (-1,0) circle (2.5);
                    \draw (-2, -3.7) node[anchor=south] {\Large $i_3(C_V)$};
                    
                    \draw[gray,ultra thin] (1,0) circle (3);
                    \draw[gray,ultra thin] (1,0) circle (2.5);
                    \draw (2, -3.7) node[anchor=south] {\Large $i_3(C_H)$};
                    
                    \draw[gray, dashed] (1.3,-0.2) ++ (80:3.7) arc (80:-80:3.7) -- ++(-4,0) arc (260:100:3.7) -- ++(4,0);
                    \draw (-4,2.8) node[anchor=south] {\Large $\tilde{\alpha}$};
                    
                    \filldraw[fill=white,draw=black,thick]
                        (-1.5, 3) ++(54:2.8) arc (54:-54:2.8) --
                        ++(126:0.6) arc (-54:54:2.2) -- cycle;
                    \filldraw[fill=white,draw=black,thick] (0, 5) circle (0.3);
                    \filldraw[fill=white,draw=black,thick] (0, 1) circle (0.3);
                    
                    \draw[gray, ultra thin] (-1.5, 3) ++(-20:2.5) ++(-20:0.3) arc (-20:-200:0.3);
                    \draw[gray, dashed] (-1.5, 3) ++(-20:2.5) ++(160:0.3) arc (160:-20:0.3);
                    \draw[gray, ultra thin] (-1.5, 3) ++(-22:2.5) ++(-22:0.3) arc (-22:-202:0.3);
                    \draw[gray, dashed] (-1.5, 3) ++(-22:2.5) ++(158:0.3) arc (158:-22:0.3);
                    
                    \draw (1.2, 2) node[anchor=west] {\Large $i_3(C_2)$};
                    
                    \filldraw[fill=white,draw=black,thick]
                        (-4.2, 3) ++(54:2.8) arc (54:-54:2.8) --
                        ++(126:0.6) arc (-54:54:2.2) -- cycle;
                    \filldraw[fill=white,draw=black,thick] (-2.7, 5) circle (0.3);
                    \filldraw[fill=white,draw=black,thick] (-2.7, 1) circle (0.3);
                    
                    \draw[gray, ultra thin] (-4.2, 3) ++(-20:2.5) ++(-20:0.3) arc (-20:-200:0.3);
                    \draw[gray, dashed] (-4.2, 3) ++(-20:2.5) ++(160:0.3) arc (160:-20:0.3);
                    \draw[gray, ultra thin] (-4.2, 3) ++(-22:2.5) ++(-22:0.3) arc (-22:-202:0.3);
                    \draw[gray, dashed] (-4.2, 3) ++(-22:2.5) ++(158:0.3) arc (158:-22:0.3);
                    
                    \draw (-1.5, 2) node[anchor=west] {\Large $i_3(C_1)$};
                    
                    \filldraw[fill=white,draw=black,thick]
                        (1.2, 3) ++(54:2.8) arc (54:-54:2.8) --
                        ++(126:0.6) arc (-54:54:2.2) -- cycle;
                    \filldraw[fill=white,draw=black,thick] (2.7, 5) circle (0.3);
                    \filldraw[fill=white,draw=black,thick] (2.7, 1) circle (0.3);
                    
                    \draw[gray, ultra thin] (1.2, 3) ++(-20:2.5) ++(-20:0.3) arc (-20:-200:0.3);
                    \draw[gray, dashed] (1.2, 3) ++(-20:2.5) ++(160:0.3) arc (160:-20:0.3);
                    \draw[gray, ultra thin] (1.2, 3) ++(-22:2.5) ++(-22:0.3) arc (-22:-202:0.3);
                    \draw[gray, dashed] (1.2, 3) ++(-22:2.5) ++(158:0.3) arc (158:-22:0.3);
                    
                    \draw (3.9, 2) node[anchor=west] {\Large $i_3(C_3)$};
                    
                    \filldraw[black] (0,-2.55) circle (2pt);
                    \draw (0,-2.8) node[anchor=north] {\Large $i_3(s_0)$};
                \end{tikzpicture}}
                \caption{The embedding $i_3$ and a loop $\tilde{\alpha}$; this figure is embedded in $S^2$.}
                \label{fig:handles_gen_3}
            \end{minipage}%
            \begin{minipage}{.5\textwidth}
                \centering
                \pgfdeclarelayer{bg}
                \pgfsetlayers{bg,main}
                \scalebox{0.8}{
                \begin{tikzpicture}
                    \filldraw[fill=white,draw=black,thick]
                        (-1.5, 3) ++(54:2.8) arc (54:-54:2.8) --
                        ++(126:0.6) arc (-54:54:2.2) -- cycle;
                    \filldraw[fill=white,draw=black,thick] (0, 5) circle (0.3);
                    \filldraw[fill=white,draw=black,thick] (0, 1) circle (0.3);
                    
                    \filldraw[fill=white,draw=black,thick]
                        (-4.2, 3) ++(54:2.8) arc (54:-54:2.8) --
                        ++(126:0.6) arc (-54:54:2.2) -- cycle;
                    \filldraw[fill=white,draw=black,thick] (-2.7, 5) circle (0.3);
                    \filldraw[fill=white,draw=black,thick] (-2.7, 1) circle (0.3);
                    
                    \filldraw[fill=white,draw=black,thick]
                        (1.2, 3) ++(54:2.8) arc (54:-54:2.8) --
                        ++(126:0.6) arc (-54:54:2.2) -- cycle;
                    \filldraw[fill=white,draw=black,thick] (2.7, 5) circle (0.3);
                    \filldraw[fill=white,draw=black,thick] (2.7, 1) circle (0.3);
                    
                    \filldraw[black] (0,-2.55) circle (2pt);
                    \draw (0,-2.8) node[anchor=north] {\Large $i_3(s_0)$};
                    
                    \node (origin) at (0,-2.55) {};
                    
                    \begin{scope}
                    [gray,dashed,decoration={markings,mark=at position 0.7 with {\arrow[scale=2]{>}}}]
                        \draw
                            (origin) to[out=130,in=270] (-1,3) to[out=90,in=145,distance=100] (-2.7,5) arc (54:-54:2.5) to[out=226,in=160] (origin)
                            decorate{(-2.7,5) arc (54:-54:2.5) to[out=226,in=160] (origin)};
                        \draw (-1.9,-0.7) node[anchor=east] {\Large $g_2$};
                        \draw
                            (origin) to[out=55,in=270] (1.5,3) to[out=90,in=145,distance=100] (0,5) arc (54:-54:2.5) to[out=226,in=90] (origin)
                            decorate{(0,5) arc (54:-54:2.5) to[out=226,in=90] (origin)};
                        \draw (0,0) node {\Large $g_4$};
                        \draw
                            (origin) to[out=0,in=270] (5,3) to[out=90,in=145,distance=100] (2.7,5) arc (54:-54:2.5) to[out=226,in=45] (2,-1) to[out=225,in=30] (origin)
                            decorate{(2.7,5) arc (54:-54:2.5) to[out=226,in=45] (2,-1) to[out=225,in=30] (origin)};
                        \draw (2.4,-1) node {\Large $g_6$};
                    \end{scope}
                    
                    \begin{pgfonlayer}{bg}
                        \begin{scope}
                        [gray,decoration={markings,mark=at position 0.4 with {\arrow[scale=2]{>}}}]
                            \draw[postaction={decorate}] (origin) to[out=180,in=290] (-4,0) to[out=110,in=90,distance=50] (-2,1) to[out=270,in=140] (origin);
                            \draw (-3.5,0) node {\Large $g_1$};
                            \draw[postaction={decorate}] (origin) to[out=120,in=250] (-0.5,1) to[out=70,in=90,distance=50] (1,1) to[out=270,in=70] (origin);
                            \draw (-0.5,1.8) node {\Large $g_3$};
                            \draw[postaction={decorate}] (origin) to[out=40,in=290] (2,1) to[out=110,in=90,distance=50] (3.5,1) to[out=270,in=0] (origin);
                            \draw (2.5,2) node {\Large $g_5$};
                        \end{scope}
                    \end{pgfonlayer}
                \end{tikzpicture}}
                \caption{The generators of $\pi_1(\Sigma_3,i_3(s_0))$, dashed lines only for visibility.}
                \label{fig:gens_gen_3}
            \end{minipage}
        \end{figure}
        
        As can be seen in Figure~\ref{fig:gens_gen_3}, the induced homomorphism $(i_3)_*: F_3 \to \pi_1(\Sigma_3,i_3(s_0))$ acts on the generators $a,b,c$ of $F_3$ so:
        \begin{myequation}
            a \mapsto g_1 g_3 ,\\
            b \mapsto g_3 g_5 ,\\
            c \mapsto g_3 ,
        \end{myequation}
        where $a,b,c \in \pi_1(C,s_0)$ are as defined in Subsection~\ref{sec:3.1}.
        
        The following claim is key in proving Theorems~\ref{thm:2},~\ref{thm:3} in genus 3:
        
        \begin{claim}
        \label{cl:incompressible}
            The map $i_3 \restriction_C$ is incompressible, i.e. $(i_3 \restriction_C)_*$ and $\tau_{i_3 \restriction_C}$ are injective.
        \end{claim}
        
        Define the egg-beater maps on $E$ as follows:
        \begin{myequation}
            \Theta_k: F_2 \to Ham_c(E) ,\\
            V \mapsto (e_3 \circ e_2 \circ c_V)_* f_k \bigsqcup -(e_3 \circ c_1)_* f_k \bigsqcup -(e_3 \circ c_2)_* f_k \bigsqcup 0 \restriction_{C_3} ,\\
            H \mapsto (e_3 \circ e_2 \circ c_H)_* f_k \bigsqcup 0 \restriction_{C_1} \bigsqcup -(e_3 \circ c_2)_* f_k \bigsqcup -(c_3)_* f_k .
        \end{myequation}
        
        These are Hamiltonian diffeomorphisms: $\Theta_k(V)$ can be seen to be generated by the Hamiltonian
        \[
            (e_3 \circ e_2 \circ c_V)_* h_k \bigsqcup -(e_3 \circ c_1)_* h_k \bigsqcup -(e_3 \circ c_2)_* f_h \bigsqcup 0 \restriction_{C_3} ,
        \]
        and $\Theta_k(H)$ by the Hamiltonian
        \[
            (e_3 \circ e_2 \circ c_H)_* f_k \bigsqcup 0 \restriction_{C_1} \bigsqcup -(e_3 \circ c_2)_* f_k \bigsqcup -(c_3)_* f_k .
        \]
        
        The egg-beater maps on $\Sigma_3$ are $(i_3)_* \Theta_k(w)$, for $w \in F_2$.
    
    \subsection{Incompressibility of the embedding}
        The incompressibility of $i_3 \restriction_C$ is shown using a concept similar to the  intersection number between curves. The intersection number of free homotopy classes of a surface is the minimum number of intersections between any two loops representing the classes, see Subsection 1.2.3 of~\cite{FM} for a precise definition. In this subsection, the next similar definition is used.

        \begin{definition}
            Given a surface $M$, a basepoint $p \in M$, and a curve $\tilde{\alpha}: S^1 \to M$, define
            \begin{myequation}
                int_{M,p,\tilde{\alpha}}: \pi_1(M,p) \to \Z ,\\
                \gamma \mapsto \min_{\tilde{\gamma}} |\tilde{\gamma} \cap \tilde{\alpha}| ,
            \end{myequation}
            where the minimum runs over all loops $\tilde{\gamma}: S^1 \to M$ such that $\tilde{\gamma}$ is a representative of $\gamma$ with basepoint $p$, and where $|\tilde{\gamma} \cap \tilde{\alpha}|$ denotes the number of intersections of $\tilde{\gamma},\tilde{\alpha}$.
        \end{definition}    
        A loop $\tilde{\gamma}: S^1 \to M$ based at $p \in M$ is said to be in \textit{minimal position} with respect to $\tilde{\alpha}$ if $|\tilde{\gamma} \cap \tilde{\alpha}| = int_{M,p,\tilde{\alpha}}([\tilde{\gamma}]_{\pi_1(M,p)})$. A loop $\tilde{\gamma}$ in a surface $M$, based in $p$, is said to form a \textit{free bigon} with another loop (which is not necessarily based in $p$), $\tilde{\alpha}$, if there is an embedded closed disk $D \subset M$ such that $\partial D$ is the union of an arc of $\tilde{\alpha}$ and the image under $\tilde{\gamma}$ of an arc of $S^1$ that does not contain 0 (here $S^1 \simeq \R / \Z$). The following claim gives a criterion for being in minimal position with respect to $\tilde{\alpha}$:
        
        \begin{claim}
        \label{cl:bigon}
            Let $M$ be a closed orientable surface of genus $\geq 2$, $p \in M$, $\tilde{\alpha}: S^1 \to M$ be a loop which does not pass through $p$, and $\tilde{\gamma}: S^1 \to M$ be a loop based at $p \in M$. Then $\tilde{\gamma}$ is in minimal position with respect to $\tilde{\alpha}$ if and only if $\tilde{\gamma}$ does not form a free bigon with $\tilde{\alpha}$.
        \end{claim}
        
        This claim is a variant of the bigon criterion, presented as Proposition~1.7 of~\cite{FM}, which checks whether two free homotopy classes are in minimal position. The proof of Claim~\ref{cl:bigon} is analogous to the proof of Proposition~1.7 of~\cite{FM}, with some extra book-keeping in order to follow whether $\tilde{\gamma}$ and $\tilde{\alpha}$ pass through $p$.
        
        Using Claim~\ref{cl:bigon}, one can prove Claim~\ref{cl:incompressible}.
        
        \begin{proof}[Proof of Claim~\ref{cl:incompressible}]
            Consider the automorphism $I: F_3 \to F_3$ given by
            \begin{myequation}
                a \mapsto ac ,\\
                b \mapsto cb ,\\
                c \mapsto c ,
            \end{myequation}
            and denote by $\varphi: F_3 \to \pi_1(\Sigma_3,i_3(s_0))$ the following homomorphism:
            \begin{myequation}
                a \mapsto g_1 ,\\
                b \mapsto g_5 ,\\
                c \mapsto g_3 .
            \end{myequation}
            Note that $(i_3 \restriction_C)_* = \varphi \circ I$. Since $I$ is an isomorphism, it preserves conjugacy classes, and so $i_3 \restriction_C$ is incompressible if and only if $\varphi$ preserves conjugacy classes, i.e. if and only if for all $x,y \in F_3$ which are not conjugates, $\varphi(x),\varphi(y) \in \pi_1(\Sigma_3,i_3(s_0))$ also are not conjugates.
            
            Recall the definitions and the statement of the Conjugacy Theorem for Free Products with Amalgamation, given in Section~\ref{sec:2}. The conjugacy relation, in any group, will again be denoted by $\sim$. Assume by contradiction that there exist $x,y \in F_3$ such that $x \not\sim y$ and $\varphi(x) \sim \varphi(y)$. In order for the concept of cyclically reduced elements of $F_3$ to be defined, consider $F_3 = \langle a,b,c \rangle$ as the free product $\left \langle \left \langle R \right \rangle * \left \langle S \right \rangle \right \rangle$, for some two proper subsets $R \bigsqcup S = \{a,b,c\}$. This partition of $\{a,b,c\}$ gives a well-defined concept of cyclically reduced sequences of elements of $F_3$, and also a well defined length function $len = len_{R,S} : F_3 \to \Z_{\geq 0}$, defined in Section~\ref{sec:2}. Without loss of generality, assume $x,y$ are cyclically reduced; this means that $\varphi(x), \varphi(y)$ are also cyclically reduced. We will reach a contradiction. \\
            
            First, consider the case where for all partitions $R,S$, $len_{R,S}(x), len_{R,S}(y) < 2$. If $len_{R,S}(x) = 0$ or $len_{R,S}(y) = 0$, then $x = 1$ or $y = 1$ respectively, and therefore $\varphi(x) = 1$ or $\varphi(y) = 1$ respectively. This implies that $\varphi(x) = \varphi(y) = 1$, and by injectivity of $\varphi$, $x = y = 1$, in contradiction. Therefore, we get that for all partitions $R,S$, $len_{R,S}(x) = len_{R,S}(y) = 1$. This implies that there exist $r, s \in \{a,b,c\}$ such that $x \sim r^n, y \sim s^m$ for some $m,n \in \Z$. By passing to the abelianization $\Z^6$ of $\pi_1(\Sigma_3)$, one sees that in fact $r^n = s^m$, so $x \sim y$, in contradiction. Therefore there exists a proper partition $R \bigsqcup S = \{a,b,c\}$ such that $len_{R,S}(x) \geq 2$ or $len_{R,S}(y) \geq 2$. By symmetry between $x$ and $y$ and between $a,b,c$, one may assume that $R = \{a\}, S = \{b,c\}$, and $len_{R,S}(x) \geq 2$.
            
            Denote the groups $G = \langle g_1,g_2 \rangle \simeq F_2$, $H = \langle g_3,...,g_6 \rangle \simeq F_4$, subgroups $A = \langle [g_1,g_2] \rangle < G$, $B = \langle [g_3,g_4][g_5,g_6] \rangle < H$, and the isomorphism $\psi: A \xrightarrow{\sim} B : [g_1,g_2] \mapsto ([g_3,g_4][g_5,g_6])^{-1}$. Note that $\pi_1(\Sigma_3) = \langle G * H, A = B, \psi \rangle$ and that $\varphi$ preserves factors, in the sense of Section~\ref{sec:2}. Let $(x_1,...,x_n)$ be a cyclically reduced sequence representing $x$ (i.e. $x = \Pi_{i=1}^n x_i$).
            
            Since $\varphi(x) \sim \varphi(y)$, by the Conjugacy Theorem for Free Products with Amalgamation, there exist $\mu \in A, 1 \leq k \leq n$ such that $\varphi(y) = \mu \cdot \Pi_{i=k}^{k-1} \varphi(x_i) \cdot \mu^{-1}$ (where $\Pi_{i=k}^{k-1} \varphi(x_i)$ is shorthand for $\varphi(x_k) \cdot ... \cdot \varphi(x_n) \cdot \varphi(x_1) \cdot ... \cdot \varphi(x_{k-1})$), and since $A = \langle [g_1,g_2] \rangle$, $\varphi(y) = [g_1,g_2]^l \cdot \Pi_{i=k}^{k-1} \varphi(x_i) \cdot [g_1,g_2]^{-l}$ for some $l \in \Z$.
            
            If $l = 0$, $\varphi(y) = \Pi_{i=k}^{k-1} \varphi(x_i)$. Since $\varphi$ is injective, this equation can be pulled back to $F_3$ to get $y = \Pi_{i=k}^{k-1} x_i \sim x$, in contradiction. Thus $l \neq 0$.
            
            Consider the loop $\tilde{\alpha}: S^1 \to \Sigma_3$, depicted in Figure~\ref{fig:handles_gen_3}. The image of $\tilde{\alpha}$ does not intersect $i_3(C)$, and its free homotopy class is $[g_1 g_3 g_5]_{\pi_0(\mathscr{L}\Sigma_3)}$. Since $\Ima(\tilde{\alpha}) \cap i_3(C) = \emptyset$ then for any element $\gamma \in \pi_1(C,s_0)$, $(i_3)_*(\gamma)$ has a representative loop in $\Sigma_3$ that does not intersect $\tilde{\alpha}$ (as such a representative can be chosen to be contained in $i_3(C)$), and thus for any element $\gamma \in \pi_1(C,s_0)$, $\varphi(\gamma)$ also has a representative that does not intersect $\tilde{\alpha}$, since $\Ima \varphi = \Ima \ (i_3)_*$. Therefore, $int_{\Sigma_3,i_3(s_0),\tilde{\alpha}}(\varphi(\gamma)) = 0$ for all $\gamma \in \pi_1(C,s_0)$. We will show that $int_{\Sigma_3,i_3(s_0),\tilde{\alpha}}(\varphi(y)) > 0$, in contradiction. \\
            
            Construct the following representative $\tilde{\upsilon}: S^1 \to \Sigma_3$ to $\varphi(y)$. Note that representatives of the element $[g_1,g_2] \in \pi_1(\Sigma_3)$ circle the leftmost handle (i.e. the one that contains $i_3(C_1)$). Take a representative $\tilde{\rho}: S^1 \to \Sigma_3$ to $[g_1,g_2]$, based in $i_3(s_0)$, that is in minimal position with respect to $\tilde{\alpha}$. Note that since $\tilde{\rho}$ circles the leftmost handle, $|\tilde{\rho} \cap \tilde{\alpha}| > 0$. Take a representative $\tilde{\tau}: S^1 \to i_3(C)$ to $\Pi_{i=k}^{k-1} \varphi(x_i)$ that is contained in $i_3(C)$, and denote $\tilde{\upsilon} = (\tilde{\rho})^l \# \tilde{\tau} \# (\tilde{\rho})^{-l}$ (where $\#$ denotes path concatenation). By reparametrization, it can be assumed that $\tilde{\upsilon} \restriction_{[0,\frac13]}, \tilde{\upsilon} \restriction_{[\frac13,\frac23]}, \tilde{\upsilon} \restriction_{[\frac23,1]}$ are parametrizations of $(\tilde{\rho})^l, \tilde{\tau}, (\tilde{\rho})^{-l}$ respectively (i.e. $\tilde{\upsilon}$ passes from $(\tilde{\rho})^l$ to $\tilde{\tau}$ at time $\frac13$, etcetera).
            
            If $\tilde{\upsilon}$ forms a free bigon with $\tilde{\alpha}$, the bigon's boundary must be made up of an arc of $\tilde{\alpha}$, which intersects $\tilde{\rho}$ at its endpoints (since $\tilde{\alpha}$ and $\tilde{\tau}$ are disjoint), and an arc $\tilde{\upsilon} \restriction_{[t_1,t_2]}$, where $[t_1,t_2] \subset S^1$ is some interval not containing 0. If $t_1 < \frac13$ and $t_2 > \frac23$, it can be seen that $\tilde{\tau}$ is contractible. This implies that $\Pi_{i=k}^{k-1} \varphi(x_i) = 1$ and therefore $x=1$, by contradiction to $n \geq 2$. Thus $t_1 \geq \frac23$ or $t_2 \leq \frac13$. But this implies that $\tilde{\rho}$ forms a free bigon with $\tilde{\alpha}$, in contradiction to the fact that $\tilde{\rho}$ is in minimal position with respect to $\tilde{\alpha}$. Thus $\tilde{\upsilon}$ cannot form a free bigon with $\tilde{\alpha}$. By Claim~\ref{cl:bigon}, $\tilde{\upsilon}$ is in minimal position with respect to $\tilde{\alpha}$, and therefore $int_{\Sigma_3,i_3(s_0),\tilde{\alpha}}(\varphi(y)) = |\tilde{\upsilon} \cap \tilde{\alpha}| > 0$. This is a contradiction, and completes the proof.
        \end{proof}

\bibliographystyle{plain}
\bibliography{refs.bib}

\end{document}